\documentclass{article}

\usepackage[left=2.5cm, right=2.5cm, top=2.5cm, bottom=3cm]{geometry}

\usepackage{graphicx}
\usepackage{caption}
\usepackage{subcaption}
\usepackage{makeidx}
\usepackage{amsfonts,dsfont}
\usepackage{amssymb}
\usepackage{amsthm}
\usepackage{amsmath}
\usepackage{textcomp}
\usepackage[singlespacing]{setspace}
\usepackage[normalem]{ulem}
\usepackage{lineno}
\usepackage{color} 
\usepackage{bm,stackrel}
\usepackage{textpos}
\usepackage{anysize}

\usepackage{tikz,pgf}
\usetikzlibrary{arrows}
\usetikzlibrary{snakes,automata,patterns}
\usepackage
{circuitikz} 

\usepackage[ruled,linesnumbered]{algorithm2e} 
\usepackage{setspace}

\RequirePackage{parskip}
\usepackage{chngcntr}
\counterwithin{figure}{section}
\counterwithin{equation}{section}

\usepackage{float}

\usepackage{lipsum}

\usepackage{pgf}
\usepackage{tikz}
\usetikzlibrary{arrows,automata}
\usepackage{pgfplots}
\pgfplotsset{width=8cm,compat=1.5}

\newcounter{defcounter}
\newtheoremstyle{miestilo}{\baselineskip}{3pt}{\itshape}{}{\bfseries}{}{.5em}{}
\newtheoremstyle{miobs}{\baselineskip}{3pt}{}{}{\bfseries}{}{.5em}{}
\theoremstyle{miestilo}
\newtheorem{defn}{Definition}

\newtheorem{thm}[defn]{Theorem}
\newtheorem{lemma}[defn]{Lemma}
\newtheorem{cor}[defn]{Corollary}

\theoremstyle{miobs}

\newcommand{\R}{\mathbb{R}}

\newcommand{\act}{\mathrm{act}}
\newcommand{\pos}{\mathrm{pos}}

\newcommand{\Cost}{\mathrm{Cost}}
\newcommand{\Coef}{\mathrm{Coef}}

\newcommand{\Cin}{\mathrm{in}}
\newcommand{\Cout}{\mathrm{out}}

\newcommand{\STC}{\mathrm{STC}}

\usepackage{tikz}
\usetikzlibrary{shapes.geometric}
\usepackage[numbers]{natbib} 

\allowdisplaybreaks
\usepackage{pdflscape}

\usepackage{listings}
\usepackage{xcolor}
\usepackage{resizegather}  
\pgfplotsset{compat=newest}
\usepgfplotslibrary{statistics}

\definecolor{RYB1}{rgb}{0.36, 0.54, 0.66}
\definecolor{RYB2}{rgb}{0.82, 0.1, 0.26}
\definecolor{RYB3}{rgb}{0.0, 0.0, 1.0}
\definecolor{RYB4}{RGB}{230,97,1}
\definecolor{RYB5}{RGB}{200,150,250}
\pgfplotscreateplotcyclelist{colorbrewer-RYB}{
	{RYB1!50!black,fill=RYB1},
	{RYB2!50!black,fill=RYB2},
	{RYB3!50!black,fill=RYB3},
	{RYB4!50!black,fill=RYB4},
	{RYB5!50!black,fill=RYB5}} 

\makeatletter \newcommand{\pgfplotsdrawaxis}{\pgfplots@draw@axis} \makeatother
\pgfplotsset{axis line on top/.style={
		axis line style=transparent,
		ticklabel style=transparent,
		tick style=transparent,
		axis on top=false,
		after end axis/.append code={
			\pgfplotsset{axis line style=opaque,
				ticklabel style=opaque,
				tick style=opaque,
				grid=none}
			\pgfplotsdrawaxis}
	}
}

\definecolor{codegreen}{rgb}{0,0.6,0}
\definecolor{codegray}{rgb}{0.5,0.5,0.5}
\definecolor{codepurple}{rgb}{0.58,0,0.82}
\definecolor{backcolour}{rgb}{0.95,0.95,0.92}

\lstdefinestyle{mystyle}{
	backgroundcolor=\color{backcolour},   
	commentstyle=\color{codegreen},
	keywordstyle=\color{magenta},
	numberstyle=\tiny\color{codegray},
	stringstyle=\color{codepurple},
	basicstyle=\ttfamily\footnotesize,
	breakatwhitespace=false,         
	breaklines=true,                 
	captionpos=b,                    
	keepspaces=true,                 
	numbers=left,                    
	numbersep=5pt,                  
	showspaces=false,                
	showstringspaces=false,
	showtabs=false,                  
	tabsize=2
}

\lstdefinelanguage{Julia}%
{morekeywords={abstract,break,case,catch,const,continue,do,else,elseif,%
		end,export,false,for,function,immutable,import,importall,if,in,%
		macro,module,otherwise,quote,return,switch,true,try,type,typealias,%
		using,while},%
	sensitive=true,%
	alsoother={\$},%
	morecomment=[l]\#,%
	morecomment=[n]{\#=}{=\#},%
	morestring=[s]{"}{"},%
	morestring=[m]{'}{'},%
}[keywords,comments,strings]%

\title{Optimal design of exchange water networks with control inputs in Eco-Industrial Parks}

\author{
	Didier Aussel \thanks{ Laboratoire PROMES, UPR CNRS 8521, Universit\'e de Perpignan - Via Domitia, 66100 Perpignan, France {\tt aussel@univ-perp.fr} (Corresponding Author) }
	\and 
	Kien Cao Van \thanks{ Laboratoire PROMES, UPR CNRS 8521, Universit\'e de Perpignan - Via Domitia, 66100 Perpignan, France {\tt kien.van@promes.cnrs.fr}  }
	\and
	David Salas \thanks{Instituto de Ciencias de la Ingenier\'ia, Universidad de O'Higgins, Rancagua, Chile {\tt david.salas@uoh.cl}}
}

\newlength\myindent
\providecommand{\keywords}[1]{\textbf{Keywords} #1}

\providecommand{\AMS}[1]{\textbf{Mathematics Subject Classification (2020)} #1}

\definecolor{forestgreenweb}{rgb}{0.13, 0.55, 0.13}
\newcommand{\DS}[1]{{{#1}}} 
\newcommand{\DA}[1]{{{#1}}} 
\newcommand{\KC}[1]{{{#1}}} 
\begin{document}
	\maketitle
	\begin{abstract}
		Industrial water conservation is an important adaptation to preserve the environment. Eco-Industrial Parks (EIPs) have been designed to encourage the establishment of water exchange networks between enterprises in order to minimize freshwater consumption and wastewater discharge by maximizing wastewater reuse. \KC{This control-input model presents a mathematical programming formulation for designing and optimizing industrial water networks in EIPs, formulating and solving it as a Single-Leader Multi-Follower (SLMF) game problem}. Enterprises (followers) aim to minimize their operating costs by reusing wastewater from other enterprises, while the designer (leader) aims to minimize the consumption of natural resources within the ecopark. Moreover, when participating \KC{in} the ecopark, enterprises can control all their input fluxes and the designer guarantees a minimal relative improvement in comparison with the stand-alone operation of each enterprise. The SLMF game is transformed into a single mixed-integer optimization problem. The obtained results are compared with the results of the blind-input model \cite{SALAS2020107053}.
	\end{abstract}
	
	\keywords{\noindent  Single-Leader-Multi-Follower game; Eco-industrial park; mixed integer programming; water networks }
	
	\AMS{ 90C26; 90C11; 90C90}

	
	\section{Introduction}
	The development of industrialization and urbanization along with other human activities on the environment in many countries around the world makes the earth heavily polluted. Therefore, decision makers need to have practical policies and actions to prevent these risks. Facing these issues, numerous approaches/concepts have been proposed,  in the last few decades, for protecting the global environment while increasing the economic utilities based on the concept of \textit{sustainable development}. Some concepts linked to sustainable development such as \textit{Industrial Ecology} (IE) have emerged \cite{Boix2015}. IE was first introduced in 1989 by Frosch
	and Gallopoulos \cite{Frosch1989}. They proposed that resource consumption and waste generation are minimized by allowing the waste materials from one industry to serve as raw material for another. This idea is directly related to another concept, namely \textit{Industrial Symbiosis} (IS), which involves ``separate industries in a collective approach to competitive advantage involving physical exchange of materials, energy, water and by-products" \cite{Chertow2000}. Industrial symbiosis can be realized through single-industry dominated complexes including chemicals and petrochemicals, pulp and paper,   waste management which includes material reuse and recycle, etc. As a feature of IS, the geographical proximity between participating industries is essential because the transportation of waste materials is expensive. The most widespread manifestations of these kinds of IS  are \textit{Eco-Industrial Parks} (EIPs).
	
	The objective of EIPs is to reduce the production cost of each participating enterprise and the environmental impact of industrial production  while the production level of each industry is maintained. This involves reducing the consumption of energy and/or raw materials (water, energy-steam, etc) of a group of companies located in the same industrial park, or designing/creating new industrial parks incorporating these aspects. This is achieved by reusing the waste from one industrial process as a utility from another process, either in raw form if the ``contamination" is low enough or via regeneration facilities. However, to convince companies to take part in an EIP, it is essential to make sure that each participant gains in competitiveness (reduction in production costs in most cases). Since these advantages depend on the EIP configuration, proper planning and design are critical. However, the system methods for designing the optimal configuration of EIP are lacking and the models of water exchange network used to model Eco-Industrial Parks in the literature are limited in some respects \cite{SALAS2020107053, Boix2015, Chertow2000, Ramos2016}.

	The design and optimization of water-exchange networks in EIPs are complex problems due to their actual sizes, their range of physical constraints, to be taken into account.
	In the literature, there are two main approaches for designing and optimizing water-exchange networks in the EIP: multi-objective optimization (MOO) on one hand and non-cooperative game theory on the other hand. The MOO perspective is based on coalition cooperative games in which enterprises make binding agreements to coordinate their strategies and share their information. The solution of such an approach is called \textit{Pareto optimal}. A solution is said to be Pareto optimal if no one can be made better off without making someone worse off \cite{mccain2010game}. After solving the multi-objective optimization problem the designer selects one of Pareto optimal solutions considering an auxiliary criterion, usually the distance to a utopia point. The basic obstacle to EIP is the need for cooperation between enterprises with different interests, in particular by sharing data on their production processes. This kind of cooperation between enterprises can only be implemented when there is trust between partners. Note that such an issue usually does not exist in designing and optimizing of exchange networks in EIPs \cite{SchwarzSteininger1997}. Therefore, due to the non-cooperative context, the different enterprises may deviate from the selection of the designer since they may improve their benefits by unilaterally change their operation. We refer the reader to \cite{Boix2015, BoixMontastruc2011, Montastruc2013}, for the survey on multi-objective optimization approach.
	
	Another suitable approach for designing and \KC{optimizing water-exchange networks} in EIPs is the non-cooperative game theory approach. A game is non-cooperative if the participants do not make binding commitments to coordinate their strategies. The theory of non-cooperative games corresponds to a mathematical analysis of strategy and conflict, in which a player success in making choices depends on the choice of others. In fact, an EIP can be considered as a collection of non-cooperative agents by introducing an EIP impartial authority whose role is to collect all necessary data, in a confidential way, to design the EIP. More precisely, enterprises optimize their operating costs while the EIP authority aims to minimize resource consumption. This kind of problem can be modeled as Single-Leader-Multi-Follower (SLMF) game with leader-follower strategy. The upper-level decision-maker (leader) is the EIP authority and the enterprises are the lower-level decision-makers (followers). The EIP authority makes his decision first by anticipating the responses of enterprises. Based on the EIP authority’ decisions, all enterprises compete with each other in \KC{a parametric generalized Nash game} at the lower-level with the strategies of the EIP authority as exogenous parameters. It's worth mentioning that, at the lower level, enterprises play a generalized Nash equilibrium between them, so enterprises involved would be able to keep confidential data, without the need to share them with the other enterprises of the park. In the context of non-cooperative games, an optimal solution for EIP design can be achieved and proposed by obtaining a generalized Nash equilibrium \KC{(GNEP)}. Due to the Nash equilibrium concept, no enterprise can unilaterally deviate in order to improve his payoff by choosing a different strategy. We refer the reader to \cite{Ichiishi1983} for a primer in non-cooperative games, to \cite{FacchineiKazow2010} for a survey of generalized Nash equilibrium problems. For SLMF games, we refer to \cite{MingFukushima2015,Aussel_Svensson_JNCA18,Aussel_Svensson_JOTA19,Aussel_Svensson2020} and the references therein. Figure \ref{fig:SLMF-Game} shows the general scheme of such a model.

	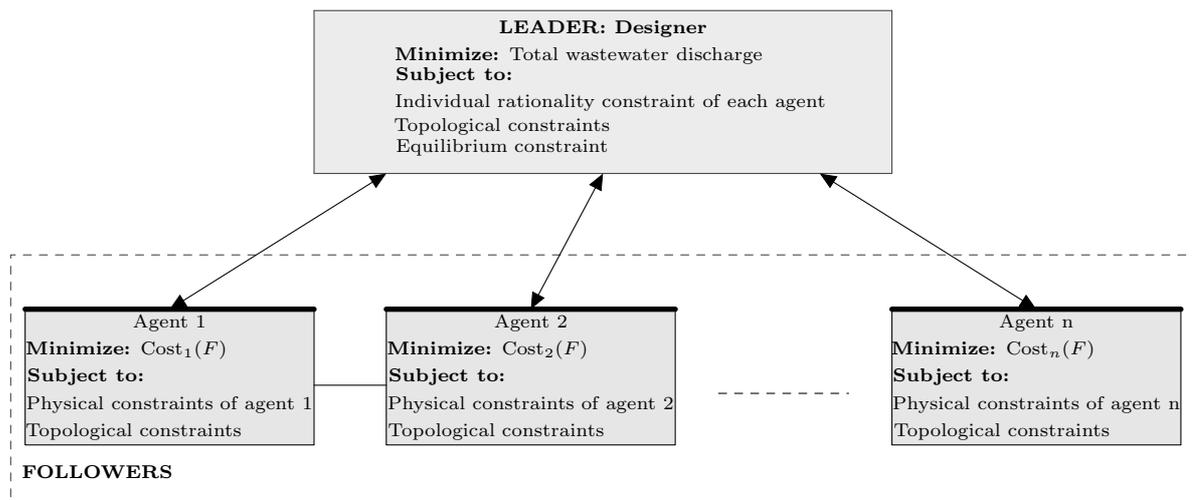
\begin{figure}[H] 
		\centering
		\definecolor{zzttqq}{rgb}{0.26666666666666666,0.26666666666666666,0.26666666666666666}
		\begin{tikzpicture}[line cap=round,line join=round,>=triangle 45,x=1.9cm,y=1.8cm]
			\fill[line width=2.pt,color=zzttqq,fill=zzttqq,fill opacity=0.10000000149011612] (-1.5,5.2) -- (-1.5,4.) -- (2.5,4.) -- (2.5,5.2) -- cycle;
			\fill[fill=black,fill opacity=0.10000000149011612] (-3.5,3.) -- (-3.5,2.) -- (-1.5,2.) -- (-1.5,3) -- cycle;
			\fill[fill=black,fill opacity=0.10000000149011612] (-1.,3.) -- (-1.,2.) -- (1.,2.) -- (1.,3.) -- cycle;
			\fill[fill=black,fill opacity=0.10000000149011612] (2.5,3.) -- (2.5,2.) -- (4.5,2.) -- (4.5,3.) -- cycle;
			\draw[dashed,color=zzttqq] (-3.6,3.4) -- (-3.6,1.56) -- (4.59,1.56) -- (4.59,3.4) -- cycle;
			
			\draw [color=zzttqq] (-1.5,5.2)-- (-1.5,4.);
			\draw [color=zzttqq] (-1.5,4.)-- (2.5,4.);
			\draw [color=zzttqq] (2.5,4.)-- (2.5,5.2);
			\draw [color=zzttqq] (2.5,5.2)-- (-1.5,5.2);
			\draw (-3.5,3.)-- (-3.5,2.)-- (-1.5,2.) -- (-1.5,3.)-- cycle;
			\draw [line width=2.pt] (-3.5,3.)-- (-1.5,3);
			\draw (-1,3.)-- (-1,2.)-- (1,2.)-- (1,3.)--cycle;
			\draw [line width=2.pt] (-1,3.)-- (1,3);
			\draw (2.5,3.)-- (2.5,2.)-- (4.5,2.)-- (4.5,3.)--cycle;
			\draw [line width=2.pt] (4.5,3.)-- (2.5,3.);
			
			\draw [<->] (-2.5,3.)-- (-1.,4.);
			\draw [<->] (0,3.)-- (0.5,4.);
			\draw [<->] (3.5,3.)-- (2.,4.);
			\draw  (-1.,2.44)-- (-1.5,2.44);
			\draw [dashed] (1.3,2.38)-- (2.2,2.38);
			\begin{scriptsize}
				\draw[color=black] (0.5,5.07) node {\textbf{LEADER: Designer}};
				\draw[color=black] (0.33,4.87) node {\textbf{Minimize:} Total wastewater discharge};
				\draw[color=black] (-0.53,4.72) node {\textbf{Subject to:}}; 
				\draw[color=black] (0.55,4.52) node {Individual rationality constraint of each agent};
				\draw[color=black] (-0.2,4.35) node {Topological constraints};
				\draw[color=black] (-0.2,4.19) node {Equilibrium constraint};
				\draw[color=black] (-3,1.8) node {\textbf{FOLLOWERS}};
				\draw[color=black] (-2.5,2.9) node {Agent 1};
				\draw[color=black] (-2.8,2.7) node {\textbf{Minimize:} $\Cost_{1}(F)$};
				\draw[color=black] (-3.08,2.5) node {\textbf{Subject to:}};
				\draw[color=black] (-2.5,2.3) node {Physical constraints of agent 1}; 	\draw[color=black] (-2.75,2.1) node {Topological constraints};
				
				\draw[color=black] (0,2.9) node {Agent 2};
				\draw[color=black] (-0.3,2.7) node {\textbf{Minimize:} $\Cost_{2}(F)$};
				\draw[color=black] (-0.58,2.5) node {\textbf{Subject to:}};
				\draw[color=black] (0,2.3) node {Physical constraints of agent 2}; 
				\draw[color=black] (-0.24,2.1) node {Topological constraints};
				
				\draw[color=black] (3.5,2.9) node {Agent n};
				\draw[color=black] (3.2,2.7) node {\textbf{Minimize:} $\Cost_{n}(F)$};	\draw[color=black] (2.91,2.5) node {\textbf{Subject to:}};
				\draw[color=black] (3.5,2.3) node {Physical constraints of agent n}; 
				\draw[color=black] (3.26,2.1) node {Topological constraints};
			\end{scriptsize}
		\end{tikzpicture}
		\caption{General scheme of SLMF Game} \label{fig:SLMF-Game}
	\end{figure}
	
	It is clear that the latter approach is more realistic in designing and optimizing the water exchange network in the EIP because first, it helps to reduce the overall freshwater consumption; second, it reduces the operating costs of each enterprise; third, enterprises operating in the EIP do not need to share their information with other enterprises of the park, which is clearly a very important issue in the design of an optimal EIP. The SLMF approach of the optimal design of EIP has been introduced for the first time in \cite{Ramos2016} and then specialized in \cite{SALAS2020107053, BoixMontastruc_et_al_2018}). In fact in \cite{SALAS2020107053}, the authors have profoundly revisited the SLMF approach of \cite{Ramos2016} and have developed an abstract Blind-Input model for water exchange networks in EIPs. The main implicit assumption done in the blind-input model \cite{SALAS2020107053} is that each enterprise can only control his outlet distribution. They are thus forced to accept whatever is sent to them through the exchange network. This assumption seems actually quite restrictive since the enterprise may be forced to receive too much polluted water which could turn into higher costs than the stand-alone operation outside the park. However, to overcome this obstacle, the authors came up with the concept of a Blind-Input contract, which guarantees that the designer commits a minimal relative improvement of his operating cost, with respect to the stand-alone operation of the enterprise.
	
	In the present work, we propose another model for designing and optimizing exchange water networks by considering that each enterprise controls his input flux, which seems more realistic in the case of exchange networks. In other words, when participating \KC{in} the exchange network, each enterprise has the ability to control the amount of water coming from the other enterprises. This model is called \textit{control-input model}.\DS{ We show that} under some linear structure of the costs functions  $\Cost_{i}(\cdot)$ of each enterprise, the control-input model can be \DS{solved by finding the solution of an auxiliary single-level mixed-integer problem. This auxiliary problem can be approximated by a family of usual single-level mixed-integer linear problems, which can be treated by common commercial solvers.} The approach is validated on a case study of exchange water network in EIPs without regeneration units. Obtained results are compared with the blind-input model \cite{SALAS2020107053}.
	
	The rest of this paper is organized as follows. Section \ref{sec:EIP} provides the general problem statement which briefly describes the problem addressed in this article and present a control input model for water exchange networks in EIPs, based on a single-leader-multiple-follower model. \DS{The construction of the auxiliary problem, its approximations, and the mathematical arguments that validate it as an alternative to find solutions of the SLMF problem are developed in} Section \ref{sec:MIPR}.
	Numerical experiments on reasonably large EIP are presented in Section \ref{sec: numerical results} where comparisons with blind input approach are also provided.
	Finally, Conclusions and perspectives are presented in Section \ref{sec:conclusion}.
	
	\section{Control-Input model for water exchange networks in EIPs}
	\label{sec:EIP}
	
	\subsection{Problem statement}\label{subsec:Problem_statement}
	Let us consider a set of enterprises \KC{$\mathcal{P} :=\{ \mathcal{P}_1,\ldots, \mathcal{P}_n \}$} that are co-located in the same industrial park and are governed by rules made in the park. Each enterprise has its own pre-defined water input requirement and quality characteristics, as well as the quantity and quality of available output wastewater. After the operation of each enterprise, the discharge wastewater can be used as input for other enterprises in the park. Contrary to the model developed in \cite{SALAS2020107053} and \cite{Ramos2016}, our {\em Control-Input} model is based on the assumption that each enterprise has a control on his polluted inputs, namely, each enterprise has the ability to set the amount of water coming from the other enterprises. The goal of the model is to establish an optimal exchange network so that the total freshwater consumption, the total discharge wastewater, and the operating cost of each participating enterprise in the park are minimized, while satisfying all process and environmental constraints. The problem is structured as a SLMF problem wherein the EIP authority is the upper-level decision-maker and the enterprises interact through a GNEP as lower-level decision-makers. Each enterprise wants to minimize his cost of the use of water while the designer is in charge of the ecological concern by minimizing the fresh water consumption, thus encouraging the recycling or reuse of wastewater streams.
	
	The $n$ enterprises of the considered EIP are connected thanks to an exchange network. A sink node, represented by the index $0$ is of course included. It represents a waste pit to discharge polluted water. Each enterprise $\mathcal{P}_i$ can be connected to other enterprises and/or to the sink node. Setting by $I:=\{ 1,\ldots,n \}$ the index set of the $\mathcal{P}$ we also define $I_0 = \{ 0\}\cup I$. \DS{In what follows, we will use $n = |I|$ to denote the number of enterprises in the park.}
	
	\subsection{Enterprises' problem}\label{subsec:Enterprises}
	
	The goal of the designer is to built an exchange network so that part of this polluted water could be reused by other enterprises, reducing the global consumption of fresh water within the park. Here, an exchange network for the EIP is a simple directed graph $(I_0,E)$, where the connection $(i,j)\in E$ means that the enterprise $i$ can send its output water to the enterprise $j$.  In this sense, if the enterprise $i$ uses the connection $(i,0)$, then it means that it is discharging water outside the park, to the environment. 
	
	\DS{Defining the sets
		\begin{align*}
			E^{\rm st} &:=\{ (i,0)\, :\, i\in I  \},\\
			E^{\max} &:= \{ (i,j)\ :\ i\in I,\, j\in I_0 \}
		\end{align*}
		
		a valid exchange network $E$ is the set of connections satisfying $E^{\rm st}\subset E\subset E^{\max}$.} Note that the set $E^{\rm st}$ is the \emph{stand-alone configuration}, where each enterprise only has access to fresh water and, after using it, he must discharge it to the sink node. On the other hand $E^{\max}$ stands for the complete park, in the sense that all enterprises are connected between them, and each enterprise has a connection with the sink node. This definition yields that: 1) for every enterprise there is always the possibility of discharge; and 2) the sink node doesn't have any exit connections (it is not possible to recover water once it is discharged). We denote by $\mathcal{E}$ the family of valid networks for the EIP. \DS{Finally, for any $E\in\mathcal{E}$, we define its complement $E^c$  the set of connections that are not in $E$, that is, $E^c = E^{\max} \setminus E$}.
	
	For each $(k,i)\in E^{\max} $ we denote by $F_{k,i}$ [T/h] the water flux going from $k$ to $i$ through the connection $(k,i)$. We consider the following notation:
	\begin{itemize}
		\item $F_i = ( F_{k,i}\ :\ k \in I )$ is the vector of fluxes going to enterprise $i$,
		\item \KC{$F_{-i} = ( F_{k,j}\ :\ k \in I, j\in I\setminus \{i\})$} is the vector of all fluxes not going to enterprise $i$ and not going to sink node,
		\item $F =  ( F_i\ :\ i\in I )$ is the vector of fluxes between enterprises,
		\item for any $i\in I$, $z_i$ stands for the inlet fresh water and $F_{i,0}$ is the flux sent by enterprise $i$ to the sink node.
	\end{itemize}
	
	Finally, following classical notations of game theory, for an enterprise $i\in I$, we may write $F = (F_i,F_{-i})$, to stress the vector of fluxes between enterprise $i$. Then, for a fixed network $E$, \KC{valid flux vectors} $F= (F_{k,i}\ :\ (k,i)\in E^{\max} )$ and $F_0=(F_{i,0}\ :\ i\in I)$ must satisfy the following constraints:
	\begin{enumerate}
		\item {Use of connections in $E$:} since $E$ represents the available connections, we must put
		\begin{equation}\label{eq:NullFluxes}
			\forall (k,i)\in E^c,~~ F_{k,i} = 0.
		\end{equation}
		\item {Positivity of fluxes:} all the fluxes in the park must be positive:
		\begin{equation}\label{eq:Positivity-fluxes-EIP}
			\forall (k,i)\in E,\, F_{k,i}\geq 0\qquad\mbox{and}\qquad \forall i \in I,\, z_i\geq 0. 
		\end{equation}
		\item {Water mass balance:} since no water losses are considered, for any valid network $E$ and for each enterprise $i \in I$, we have 
		\begin{equation} \label{eq:watermass_balance}
			z_{i} +  \sum_{(k,i)\in E} F_{k,i}  =  \sum_{(i,j)\in E} F_{i,j} + F_{i,0}.
		\end{equation} 
	\end{enumerate}
	
	It is important to note that while each enterprise $i$ controls his inlet \KC{flux} $F_i$ and the outlet fluxes $F_{-i}$ are under the control of the concurrent enterprises, the fluxes $(F_{i,0})_{i\in I}$ corresponding to the amounts of discharged water are directly deduced by the water balance constraint \eqref{eq:watermass_balance} and will not be variables neither of the enterprises nor of the designer.
	
	\DS{To simplify notation, in what follows we will write for each $i\in I$
		\begin{align*}
			\sum_j F_{i,j} := \displaystyle\sum_{j\in I\setminus\{i\}} F_{i,j},\\
			\sum_k F_{k,i} := \displaystyle\sum_{k\in I\setminus\{i\}} F_{k,i}.
		\end{align*}
		That is, $\sum_j F_{i,j}$ stand for the aggregated outlet flux of $i$ sent to other agents $j\in I\setminus\{i\}$, and $\sum_k F_{k,i}$ stand for the aggregated inlet flux of $i$  received from other agents \KC{$k\in I\setminus\{i\}$}. Of course, these values are affected by the chosen network $E\in \mathcal{E}$, but they can be defined independently.}

	For an given couple $(E,F)$ of valid network and associated fluxes, an enterprise will be:
	\begin{itemize}
		\item in {\em semi-stand-alone} situation if, for any $k\in I$, $F_{k,i}=0$.
		\item in {\em stand-alone} situation if, for any $k,\,j\in E$, $F_{k,i}=0$ and $F_{i,j}=0$.
		\item \DS{in {\em active} situation, otherwise. That is, $i\in I$ is active when $\sum_{k}F_{k,i}>0$.}
	\end{itemize}
	
	Each enterprise $i\in I$ thus receives the fluxes from other enterprises within the EIP. Nevertheless for technical constraints on the process $\mathcal{P}_i$, the pollutant concentration delivered by the other enterprises cannot exceed a certain maximum value denoted here by $C_{i,\Cin}$ [ppm]. We assume here that \KC{each enterprise} $i\in I$, in order to use less fresh water, will actually accept a maximum of polluted water. The inlet flux is then mixed with the purchased fresh water $z_i$ generating therefore the inlet flux \DS{$z_{i} + \sum_k F_{k,i}$}. On the other hand enterprise $i$ generates a fixed amount of pollutant $M_i$ \KC{[g/h]}, coming from his internal production process. This pollutant is then diluted into the outlet water flux \DS{$\sum_j F_{i,j}+F_{i,0}$}, for which is usually assumed, in the design of EIP, that the pollutant concentration is less than a maximum value $C_{i,\Cout}$ [ppm]. Actually considering that enterprise $i$ will optimize his process, we will assume, as it is classically done for EIP the so-called \emph{hypothesis of optimal response} that each enterprise $i\in I$ consumes exactly the fresh water it needs to attain $C_{i,\Cout}$, and therefore, its output pollutant concentration is always equal to this constant. Obviously we have that $0\DS{\leq} C_{i,\Cin}< C_{i,\Cout}$.
	This functioning is illustrated in Figure \ref{fig:new_Concentration}.
	\begin{figure}[H] 
		\centering
		\definecolor{zzttqq}{rgb}{0.6,0.2,0.}
		\definecolor{cqcqcq}{rgb}{0.7529411764705882,0.7529411764705882,0.7529411764705882}
		\begin{tikzpicture}[line cap=round,line join=round,>=triangle 45,x=0.8cm,y=0.7cm]
			\fill[line width=4.pt,color=zzttqq,fill=zzttqq,fill opacity=0.10000000149011612] (0.18,0.5) -- (0.18,3.5) -- (2,2) -- cycle; 
			\fill[line width=4.pt,color=zzttqq,fill=zzttqq,fill opacity=0.10000000149011612] (11.82,0.5) -- (11.82,3.5) -- (10,2) -- cycle;
			
			\draw (4,0) -- (8,0) -- (8,4) -- (4,4) -- (4,0);
			
			\draw [color=zzttqq] (0.18,0.5)-- (0.18,3.5); 
			\draw [color=zzttqq] (2,2)-- (0.18,3.5);
			\draw [color=zzttqq] (0.18,0.5)-- (2,2); 
			
			\draw [->] (-1.5,0.5) -- (0.18,0.5);
			\draw [->] (-1.5,2) -- (0.18,2);
			\draw [->] (-1.5,3.5) -- (0.18,3.5);
			
			\draw[color=black,below] (-1.5,3.5) node {$F_{1,i}$};
			\draw[color=black,below] (-1.5,0.5) node {$F_{n,i}$};
			
			\draw [->] (2,2) -- (4,2);
			\draw [->] (8,2) -- (10,2);
			
			\draw [color=zzttqq] (11.82,0.5)-- (11.82,3.5); 
			\draw [color=zzttqq] (10,2)-- (11.82,0.5); 
			\draw [color=zzttqq] (10,2)-- (11.82,3.5); 
			
			\draw [->] (11.82,3.5) -- (13.82,3.5); 
			\draw [->] (11.82,2) -- (13.82,2); 
			\draw [->] (11.82,0.5) -- (13.82,.5);
			
			\draw[color=black,below] (13.82,3.5) node {$F_{i,1}$};
			\draw[color=black,below] (13.82,.5) node {$F_{i,n}$};
			
			\draw [color=black,->] (3,5)-- (3,2.25); 
			\draw [color=black,->] (9,5)-- (9,2.25);
			\draw [color=black,->] (0.9,-0.9)-- (2.5,1.9);
			\draw [color=black,-] (0,-0.9)-- (0.9,-0.9);
			\draw [color=black,-] (11.2,-0.9)-- (9.5,2);
			\draw [color=black,->] (11.2,-0.9)-- (12.1,-0.9);	
			
			\begin{scriptsize}
				\draw[color=zzttqq] (0.9,2) node {Mixer};
				\draw[color=zzttqq] (11,2) node {\tiny Dispatcher};
				\draw[color=black] (-1.1,-0.8) node {Freshwater $z_i$};
				\draw[color=black] (3,6) node {Max inlet pollutant};
				\draw[color=black] (3,5.6) node {concentration $C_{i,\Cin}$};
				\draw[color=black] (6,2) node {Process $\mathcal{P}_{i}$};
				\draw[color=black] (8.5,6) node {Max outlet pollutant};
				\draw[color=black] (8.5,5.6) node {concentration $C_{i,\Cout}$};
				\draw[color=black] (13.3,-0.9) node {Discharge $F_{i,0}$};
			\end{scriptsize}
		\end{tikzpicture}
		\caption{Water mixture description for a given process $\mathcal{P}_i$.\label{fig:new_Concentration}}
	\end{figure}
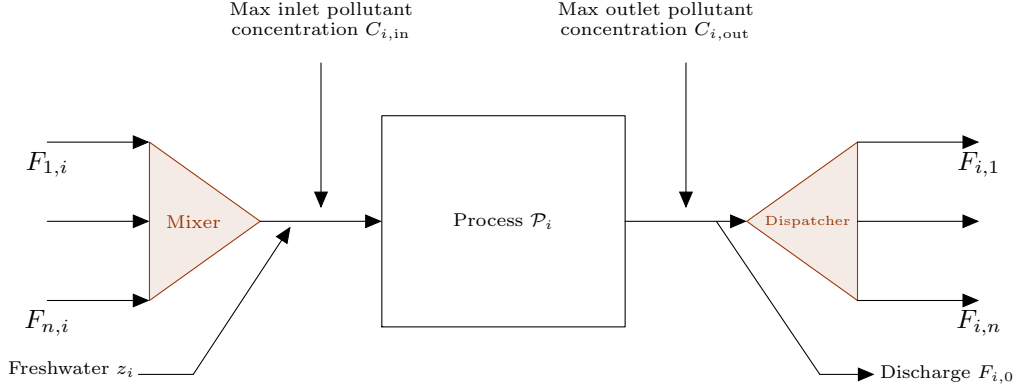
	
	Taking into account the above explanations, for any valid network $E\in\mathcal{E}$, the water mixture model immediately leads to the following constraints:
	
	\begin{enumerate}
		\item {Contaminant mass balance:} each enterprise $i \in I$ wants to receive the maximum of pollutant and thus the contaminant mass balance is
		\begin{equation}\label{eq:new_contaminant_balance}
			M_i+C_{i,\Cin}\Phi_i(z_i,(F_{k,i})_k)=C_{i,\Cout}\left(\DS{\sum_{j} F_{i,j}}+F_{i,0}\right),
		\end{equation}
		\noindent where the auxiliary function $\Phi_i:\R\times \R^{n-1}\to\R$ is defined as
		\[\KC{\Phi_{i}(z_i,(F_{k,i})_k)} = \left\{\begin{array}{ll}
			z_i+\DS{\sum_{k} F_{k,i}}  & \mbox{If there exists } \,k\mbox{ such that } F_{k,i}>0\\[5pt]
			0 & \mbox{otherwise}.
		\end{array}\right.
		\]
		Note that the auxiliary function $\Phi_i$ allows to take into account that when enterprise $i$ doesn't receive any polluted water, thus being in a semi-stand-alone or stand-alone situation, then the pollutant only comes from the internal generation $M_i$.
		\item {Inlet/outlet concentration constraints:} for an enterprise $i \in I$ we have that:
		\begin{equation} \label{eq:Inletoutlet_independent}
			\sum_{k\in I\setminus\{i\}} C_{k,\mathrm{out}} F_{k,i}  \leq   C_{i,\Cin} \left (  z_{i} + \sum_{k} F_{k,i}\right ).
		\end{equation}
	\end{enumerate}
	Observe that, in the case of an enterprise $i$ in semi-stand-alone or stand-alone situation, combining equations \eqref{eq:watermass_balance} and  \eqref{eq:new_contaminant_balance} we obtain:
	\[
	M_i=C_{i,\Cout}\left(\sum_{j} F_{i,j} + F_{i,0}\right),
	\]
	and so, the discharge of wastewater by the enterprise $i\in I$ is given by the fluxes controlled by the other enterprises, thanks to the formula
	\begin{equation} \label{eq:new_Freshwater_formula_st}
		\DS{F_{i,0}}=\frac{M_i}{C_{i,\Cout}}-\sum_{j} F_{i,j}.
	\end{equation}
	For all the other enterprises, neither in semi-stand-alone nor in stand-alone situation, by combining \eqref{eq:watermass_balance} and \eqref{eq:new_contaminant_balance}, one gets
	\[
	M_i+C_{i,\Cin}\left(\sum_{j} F_{i,j} + F_{i,0}\right)=C_{i,\Cout}\left(\sum_{j} F_{i,j} + F_{i,0}\right),
	\]
	and so, the discharge of wastewater by the enterprise $i\in I$ is given by the fluxes controlled by the other enterprises, thanks to the following formula
	\begin{equation} \label{eq:new_Freshwater_formula}
		\DS{F_{i,0}}=\frac{M_i}{C_{i,\Cout}-C_{i,\Cin}}-\sum_{(i,j)\in E} F_{i,j}.
	\end{equation}
	Mixing both equations \eqref{eq:new_Freshwater_formula_st} and \eqref{eq:new_Freshwater_formula}, we conclude that $F_{i,0}$ is given by the following explicit piecewise formula:
	\DA{\begin{equation}\label{eq:ExpressionF0-Piecewise}
			F_{i,0}(F)=\begin{cases}
				\frac{M_i}{C_{i,\Cout}}-\sum_{j} F_{i,j}\qquad&\text{ if }i \text{ operates stand-alone or semi-stand-alone},\\
				& \text{ that is if }\sum_{k} F_{k,i}=0,\\[5pt]
				\frac{M_i}{C_{i,\Cout}-C_{i,\Cin}}-\sum_{j} F_{i,j}\qquad&\text{ otherwise, that is if }\sum_{k} F_{k,i}>0.
			\end{cases}
		\end{equation}
	}
	Now we denote by $c$ [\$/T] the marginal cost of freshwater, \KC{and by $\beta$ [\$/T] the unit tax for discharged water}. Observe that, if the enterprise $i$ doesn't receive polluted water (i.e. is in semi-stand-alone or stand-alone situation), then its freshwater consumption $z_i$ must be
	\[
	z_i = \frac{M_i}{C_{i,\Cout}}.
	\]
	Then, the cost of stand-alone operation, which we denote by $\STC_i$ \KC{[\$/h]}, is given by
	\[
	\STC_i = A\cdot(c + \KC{\beta})\frac{M_i}{C_{i,\Cout}},
	\] 
	where $A$ [h] is a time constant that measures the lifetime analysis of the park.
	
	In the SLMF model, given a valid exchange network \KC{$E$}, each enterprise $i\in I$ has a cost function that he wants to minimize, defined by 
	\KC{\begin{equation} \label{eq:CostEIP}
			\displaystyle\Cost_{i}(z_i,F_i,F_{-i}, \KC{E})  =   A \Biggl[   c\cdot  z_{i}
			+   \DS{\gamma}\left(\sum_{k} F_{k,i} + \sum_{j} F_{i,j} \right) + \beta \DS{F_{i,0}(F_i,F_{-i})}\Biggr], 
		\end{equation}
		where  $\DS{\gamma}$ [\$/T] is the unit cost for the use of shared connection between nodes. In this model, we assume that the cost $c$ is much higher than $\delta$}. 
	Observe that each enterprise pays for each connection he uses, both for the inlet fluxes and outlet fluxes. In the case of the connections between two enterprises, this means that the cost of the connection is divided uniformly between the sending enterprise and the receiving enterprise.  
	
	With all these considerations, for a given valid exchange network $E\in \mathcal{E}$, the problem of each enterprise $i$ is given by problem 
	\DA{\begin{equation}\label{eq:ProblemOfIndependentAgent_EIP}
			P_i\left (z_{-i},F_{-i},E\right ) = \left\{\begin{array}{cl}
				\displaystyle \min_{z_i,F_i}\, &\Cost_i\left (z_i,F_i,F_{-i}, E\right )\\
				s.t.&\begin{cases}
					F_{k,i} = 0, ~~\forall (k,i)\in E^c,\\
					z_{i} +  \sum_{(k,i)\in E} F_{k,i}  =  \sum_{(i,j)\in E} F_{i,j} + F_{i,0}(F_i,F_{-i}),\\
					\sum_{k\in I\setminus\{i\}} C_{k,\mathrm{out}} F_{k,i}  \leq   C_{i,\Cin} \left (  z_{i} + \sum_{k} F_{k,i}\right )\\
					z_{i}\geq 0,\\
					F_i\geq 0,\\
					F_{i,0}(F_i,F_{-i})\geq 0.
				\end{cases}
			\end{array}\right.
		\end{equation}
	}
	
	Observe that constraint \eqref{eq:new_contaminant_balance} is implicit in the expressions of $F_{i,0}(F)$ given by \eqref{eq:ExpressionF0-Piecewise}.
	For a network $E\in \mathcal{E}$, we say that a vector $(z,F)$ is an equilibrium for the enterprises if and only if
	\begin{equation}\label{eq:DefEquilibria}
		\forall i\in I,\: (z_i,F_i) \mbox{ solves the problem }P_i\left(z_{-i}, F_{-i},E\right).
	\end{equation}
	We denote by $\mathrm{Eq}(E)$ the set of \KC{equilibria} for the valid exchange network $E$.
	
	\DA{It is important to quote that the existence of such equilibria cannot be guarantee here neither by classical Arrow-Debreu existence result for generalized Nash equilibrium problem (see \cite{Arrow_Debreu1954}) nor by the more recent one (see \cite{Aussel_CaoVan_Salas2021}). Indeed, given a vector $F_{-i}$, the function $F_i\mapsto F_{i,0}(F_i,F_{-i})$ described in \eqref{eq:ExpressionF0-Piecewise} is actually upper semi-continuous but not lower semi-continuous. This can be seen by rewriting this function as $F_{i,0}(F_i,F_{-i})=\sum_j F_{i,j}-\varphi_i(F_{i})$ where 
		\begin{equation}\label{eq:ExpressionF0-Piecewise_bis}
			\varphi_{i}(F_i)=\begin{cases}
				\frac{M_i}{C_{i,\Cout}}-\sum_{j} F_{i,j}\qquad&\text{ if } F_i\in \Pi,\\[5pt]
				\frac{M_i}{C_{i,\Cout}-C_{i,\Cin}}-\sum_{j} F_{i,j}\qquad&\text{ otherwise.}
			\end{cases}
		\end{equation}
		with $\Pi$ being the hyperplane \KC{$\Pi=\{F_i\in ~:~ \sum_{k} F_{k,i}=0\}$}. It is then clear that the function $\varphi_i$ is lower semi-continuous but not upper semi-continuous.
	}
	
	\subsection{Designer's problem}\label{subsec:Designer}
	The designer being in charge of the reduction of the environmental impact of the park, he wants to minimize the total wastewater discharge, that is function  
	\begin{equation}\label{eq:Totalwastewater}
		F_{0}(F) = \sum_{i \in  I} \DS{F_{i,0}(F)}.
	\end{equation} 
	\DS{By adding the water balance constraints \eqref{eq:watermass_balance} of all nodes, the total discharge water coincides with the total fresh water consumption of the park, that is, $F_0(F) = \sum_{i\in I}z_i$}. So by selecting an appropriate network $E$ and a compatible operation $(z,F)\in \mathrm{Eq}(E)$, the designer will propose an optimal design for the park with this criterium, which is usual in the literature (see, e.g., \cite{SALAS2020107053,Ramos2016} and the references therein). Nevertheless there is no a priori guarantee that doing so the resulting cost of each enterprise will be lower than his stand-alone cost $\STC_i$ and thus that every enterprise will finally accept to participate to the EIP. In order to have strong arguments to encourage enterprise to be involved into the EIP, the designer engages in a contract with them. This contract will guarantee a {\em minimal relative gain}, with respect to the stand-alone operation, denoted by $\alpha\in\, ]0,1[$, on the cost of every enterprise. 
	\begin{equation}\label{eq:BlindInputConstraint}
		\Cost_{i}(z_i,F_i,F_{-i}, E) \leq  \alpha \cdot \STC_i.
	\end{equation}
	This type of minimal relative gain contract was first introduced in \cite{SALAS2020107053}. Then, the problem of the designer is
	\begin{equation}\label{eq:ProblemDesigner}
		\begin{aligned}
			&\min_{E\in\mathcal{E},z\in\R^n,F\in\R^{|E^{\max}|}}\, F_{0}(F)\\
			&s.t.\begin{cases}
				(z,F) \in \mathrm{Eq}(E),\\
				\Cost_{i}(z_i,F_i,F_{-i}, E) \leq \alpha \cdot \STC_i,&\quad\forall i\in I.
			\end{cases}
		\end{aligned}
	\end{equation}
	
	The optimization problem \ref{eq:ProblemDesigner} can be interpreted as follows: the designer will propose to the enterprises a valid exchange network $E$, a vector $z$, and an operation $F\in\R^{|E^{\max} |}$ which satisfy all the physical constraints and also, such that the operation $F$ respects: 
	\begin{itemize}
		\item [1) ] the {\em incentive consistency}, in the sense that no enterprise will have incentives to unilaterally deviate from the proposal due to the constraint $(z, F) \in \mathrm{Eq}(\KC{E})$;
		\item [2) ] the {\em individual rationality} of each
		enterprise, in the sense that all enterprises will participate in the network
		since their participation has been bought through the constraint \eqref{eq:BlindInputConstraint}.
	\end{itemize}
	
	Let us finally bring to the fore that the designer problem corresponds to the  so-called {optimistic approach} of the SLMF game. Indeed for any given valid exchange network $E_0$, the set $\mathrm{Eq}(\KC{E})$ can contain more than one equilibrium and the designer choose the one which minimizes the function $F_0(F)$. For more details on the different possible approaches for SLMF games the reader can refer to \cite{Aussel_Svensson2020}.
	
	\subsection{Selection of participating enterprises}\label{subsec:NullClass}
	
	Physically, we know that  the network has always a feasible point, which is the Stand-Alone configuration, that is, the topology $E^{\rm st}$. However, when we include the individual rationality constraint \eqref{eq:BlindInputConstraint}, the problem may become infeasible since $\alpha<1$.
	
	Infeasibility of problem \eqref{eq:ProblemDesigner} means that the designer is not capable to find a solution that respect the contracts with all the enterprises. Thus, we need to include the possibility of excluding some enterprises from the network.  
	
	Formally, for each enterprise $i\in I$, we define a boolean variable $y_{i,\mathrm{null}}\in \{ 0,1\}$ such that
	\[
	y_{i,\mathrm{null}} = \begin{cases}
		1\qquad& \mbox{ if }i\mbox{ breaks the contract } \eqref{eq:BlindInputConstraint},\\
		0&\mbox{ otherwise.}
	\end{cases}
	\] 
	and denote by $y_{\mathrm{null}}$ the vector $y_{\mathrm{null}}=(y_{i,\mathrm{null}}\ :\ i\in I)$. With this new variable, we modify problem \eqref{eq:ProblemDesigner} adding the following constraints
	\begin{enumerate}
		
		\item  For each enterprise $i\in I$, we put
		\begin{equation}\label{eq:StandAloneNoOutlet}
			\sum_{j} F_{i,j} \leq K\cdot(1-y_{i,\mathrm{null}}),
		\end{equation}
		for some constant $K>0$ large enough. This is to ensure that, if the enterprise breaks the contract, then he cannot send his polluted water to any enterprises.
		\item For each enterprise $i\in I$, we put
		\begin{equation}\label{eq:StandAloneNoInlet}
			\sum_{k } F_{k,i} \leq K\cdot (1-y_{i,\mathrm{null}}),
		\end{equation}
		for some constant $K>0$ large enough. This constraint establishes that, if the enterprise breaks the contract, then nobody can send him any flux.
		\item For each enterprise $i\in I$, we put
		\[
		\Cost_{i}(z_i, F_i,F_{-i},\KC{E}) \leq \alpha_i \STC_i\cdot(1-y_{i,\mathrm{null}}) + \STC_i\cdot y_{i,\mathrm{null}}.
		\]
		Here, the individual rationality constraint is active only when $y_{i,\mathrm{null}} = 0$. Otherwise, since the enterprise is not connected to the network, his cost will coincide with $\STC_i$.
	\end{enumerate} 
	\DS{In this formulation, the constant $K>0$ is used as a big-M strategy: when $y_{i,\mathrm{null}} = 0$, the constraints \eqref{eq:StandAloneNoOutlet} and \eqref{eq:StandAloneNoInlet} are trivially fulfilled, since $K$ is always an upper bound for $\sum_{j} F_{i,j}$ and $\sum_{k} F_{k,i}$. A large enough value for $K$ is given by
		\begin{equation}\label{eq:ValueOfK}
			K = (|I|-1)\left(\sum_{i\in I} \frac{M_i}{C_{i,\Cout}}\right),
		\end{equation}
		
		which corresponds to the total water consumption on the EIP when every agent is stand-alone, multiplied by  the total number of agents besides enterprise $i$.} Now, denoting 
	\begin{equation}\label{eq:NewFormulaSTC}
		\STC_i(\alpha, y_{i,\mathrm{null}}) := \alpha \STC_i\cdot(1-y_{i,\mathrm{null}}) + \STC_i\cdot y_{i,\mathrm{null}},
	\end{equation}
	the new optimization problem becomes
	\begin{equation}\label{eq:ProblemOptDesign-NullClass}
		\begin{aligned}
			&\min_{E,y_{\mathrm{null}},z,F}\, F_{0}(F)\\
			&s.t.\begin{cases}
				(z,F) \in \mathrm{Eq}(E,y_{\mathrm{null}}),\\
				\displaystyle\sum_{(i,j)\in E} F_{i,j} \leq K\cdot(1-y_{i,\mathrm{null}}),\qquad&\forall i\in I,\\
				\displaystyle\sum_{(k,i)\in E} F_{k,i} \leq K\cdot (1-y_{i,\mathrm{null}}),\qquad&\forall i\in I,\\
				\Cost_{i}(z_i,F_i,F_{-i}, E) \leq \KC{\STC_i(\alpha, y_{i,\mathrm{null}})},\quad&\forall i\in I.
			\end{cases}
		\end{aligned}
	\end{equation}

	Here, $ \mathrm{Eq}(E,y_{\mathrm{null}})$ stands for the solutions of the equilibrium problem considering only the agents in
	\[
	I' = \{ i\in I\ :\ y_{i,\mathrm{null}} = 0 \},
	\]
	and forcing the agents in $I\setminus I'$ to be stand-alone. Thus, in this new problem, the designer first choose all the enterprises that will participate in the network, represented by the set 
	$I' = \{ i\in I\ :\ y_{i,\mathrm{null}} = 0 \}$ 
	and then it solves problem \eqref{eq:ProblemDesigner} replacing $I$ by $I'$. Of course, as it is formulated, the designer takes both decisions simultaneously. It is not hard to verify that any optimal solution of Problem \eqref{eq:ProblemOptDesign-NullClass} is an optimal solution of Problem  \eqref{eq:ProblemDesigner} for the reduced set of enterprises $I'$, and that Problem \eqref{eq:ProblemOptDesign-NullClass} is always feasible. We leave this verification to the reader. 
	
	\section{Mixed-Integer programming reduction}\label{sec:MIPR}
	
	The formulation of designer's problem  \eqref{eq:ProblemOptDesign-NullClass} has the form of
	a \textit{mathematical programming with equilibrium constraints} (MPEC). This section
	is devoted to prove that this MPEC formulation, which is known to be hard to solve (see, e.g., \cite{BAUMRUCKER20082903,Tseveendorj2013,luo1996mathematical}), can be sequentially approximated through a family of auxiliary single-level Mixed-Integer linear programming problems.
	
	To develop such approach, our methodology consist in construct a single optimization problem following only the objective of the designer, which is minimizing the total discharged water, without considering the equilibrium constraint. We will allow the designer to choose all the variables $(E,y_{\mathrm{null}},z,F)$, with some considerations:
	\begin{enumerate}
		\item The operation $(z,F)$ must be feasible for $(E,y_\mathrm{null})$, that is, for every agent with $y_{i,\mathrm{null}} = 0$, $(z_i,F_i)$ must be feasible for problem $P_i(z_{-i},F_{-i},E)$.
		\item The designer should include some new constraints that ensure that any solution $(E^*,y_{\mathrm{null}}^*,z^*,F^*)$ of the new problem becomes a feasible point for the original problem \eqref{eq:ProblemOptDesign-NullClass}. 
	\end{enumerate}
	
	When implementable, this methodology provides a new problem that is a relaxation in the sense that admits more feasible configurations, but a solution to this ``larger'' problem is also a solution for the original one. This section is devoted to the construction of the auxiliary problem and to proving the latter property.
	\subsection{Characterization of active agents}
	
	One of the main difficulties in order to construct the desired auxiliary problem is the discontinuity of the objective function $F_0(F)$. Indeed, recalling that for each agent $i\in I$, we have that
	\[
	F_{i,0}(F)=\begin{cases}
		\frac{M_i}{C_{i,\Cout}}-\sum_{j} F_{i,j}\qquad&\text{ if }i \text{ operates stand-alone or semi-stand-alone},\\
		\frac{M_i}{C_{i,\Cout}-C_{i,\Cin}}-\sum_{j} F_{i,j}\qquad&\text{ otherwise,}
	\end{cases}
	\]
	then  the function $F_0(F)$ has a discontinuity when an agent $i\in I$ passes from an active operation (that is, $\sum_kF_{k,i}>0$) to a semi-stand-alone operation  (that is, $\sum_kF_{k,i}=0$). If the designer wants an agent to be at stand-alone or semi-stand-alone situation, this can be done by choosing $E\in\mathcal{E}$ correctly. However, the designer cannot force an agent to be active: An agent $i\in I$ that is active for the designer's solution could change to semi-stand-alone operation if this action is profitable.
	
	In order to construct an auxiliary problem that takes this possibility into account, we need to characterize when an active agent doesn't have any interest in changing to semi-stand-alone operation. To do so, for each $i\in I$ let us define the value
	\begin{equation}\label{eq:DeltaDef}
		\Delta_i := \frac{M_i}{C_{i,\Cout} - C_{i,\Cin}} - \frac{M_i}{C_{i,\Cout}} = \frac{C_{i,\Cin}}{C_{i,\Cout}(C_{i,\Cout} - C_{i,\Cin})}M_i.
	\end{equation}
	This value represent the discontinuous reduction of discharged water when the enterprise $i$ passes from active operation to semi-stand-alone operation.

	
	\begin{lemma}[First Active agent lemma]\label{lemma:FirstActiveAgent}
		Let $E\in \mathcal{E}$ and $y_{\mathrm{null}}\in \{0,1\}^{n}$. Then, for any feasible pair $(z,F)$ for $(E,y_{\mathrm{null}})$ and any agent $i\in I$ with $\sum_kF_{k,i}>0$, we have that
		\begin{equation}\label{eq:EquivalenceDelta}
			\Cost_i(z_i,F_i,F_{-i},\KC{E})\leq \Cost_i(M_i/C_{i,\Cout},0,F_{-i},\KC{E})     \iff \sum_k F_{k,i} \geq \frac{c+\beta}{c-\gamma} \Delta_i,
		\end{equation}
		regardless if the semi-stand-alone operation $(z_i',F_i') = (M_i/C_{i,\Cout},0)$ is feasible or not for agent $i$'s problem $P_i(z_{-i},F_{-i},E)$. 	
	\end{lemma}

	\begin{proof}
		Let $(z,F)$ be a feasible pair for $(E,y_{\mathrm{null}})$, and let $i\in I$ with $\sum_kF_{k,i}>0$. Set $z_i' = M_i/C_{i,\Cout}$ and $F_i' = 0$. Then, $F_{i,0}(F_i',F_{-i}) = M_i/C_{i,\Cout} - \sum_jF_{i,j}$ (which might be negative if semi-stand-alone operation is infeasible for $i$). Then, it is not hard to see that
		\[
		F_{i,0}(F_i',F_{-i}) - F_{i,0}( F) = -\Delta_i.
		\]
		Furthermore, for each node $k\neq i$, we have that
		\[
		F_{k,0}(F_{i}',F_{-i}) = F_{k,i} + F_{k,0}(F),
		\]
		that is, the partially polluted water $F_{k,i}$ that is not consumed by $i$ must pass to be a discharged water for $k$. This yields that water balance constraint still holds for every $k\in I\setminus\{i\}$. In the case of agent $i$, water balance constraint also holds, even though the value of $F_{i,0}(F_i',F_{-i})$ might be negative. Finally, by adding the water balance constraints of all nodes, we can write
		\begin{align*}
			z_i' - z_i &= z_i' + \sum_{k\neq i} z_k - \sum_{k\in I} z_k\\
			&= F_{i,0}(F_i',F_{-i}) + \sum_{k\neq i } F_{k,0}(F_i',F_{-i}) - \sum_{k\in I } F_{k,0}(F)\\
			&= -\Delta_i + \sum_{k\neq i} (F_{k,0}(F_i',F_{-i}) - F_{k,0}(F))\\
			&= \sum_{k\in I} F_{k,i} - \Delta_i.
		\end{align*}
		Then we can write
		\begin{align*}
			&\Cost_i(z_i',F_i',F_{-i},E) - \Cost_i(z_i,F_i,F_{-i},E)\\
			=\quad& c(z_i'-z_i) + \gamma\left(\textstyle\sum_{k} F_{k,i}' - F_{k,i}\right) + \beta( F_{i,0}(F_i',F_{-i}) - F_{i,0}(F) )\\
			=\quad&c\left( \textstyle\sum_{k} F_{k,i} - \Delta_i \right) - \gamma {\textstyle\sum_{k} F_{k,i}} - \beta\Delta_i.
		\end{align*}
		By rearranging the terms, we conclude that 
		\[
		\Cost_i(z_i',F_i',F_{-i},E) - \Cost_i(z_i,F_i,F_{-i},E)\geq 0\iff \sum_{k} F_{k,i} \geq \frac{c+\beta}{c-\gamma}\Delta_i.
		\]
		The above development holds regardless if $F_{i,0}(F'_i,F_{-i}) = M_i/C_{i,\Cout} - \sum_j F_{i,j}$ is positive or not, and thus the equivalence is still valid even if semi-stand-alone operation is not feasible for agent $i$. This finishes the first the proof.
		
	\end{proof}
	
	The first Active agent lemma characterizes when semi-stand-alone is profitable for an active agent. However, our goal is to replace the equilibrium constraint for usual constraints in our auxiliary problem. With this in mind and using Lemma \ref{lemma:FirstActiveAgent}, the next lemma characterizes how must be the relation with the semi-stand-alone operation for agents in an equilibrium, in the sense of \ref{eq:DefEquilibria}.
	
	\begin{lemma}[Second Active agent lemma]\label{lemma:SecondActiveAgent}
		Let $E\in \mathcal{E}$ and $y_{\mathrm{null}}\in \{0,1\}^{n}$ and let $(z,F)\in \mathrm{Eq}(E,y_{\mathrm{null}})$. Then, for every $i\in I$ with $y_{i,\mathrm{null}} = 0$, we have that
		\begin{equation}\label{eq:ConditionsActive-notSemiSTA}
			i\in I\text{ is active} \iff \frac{M_i}{C_{i,\Cout}} - \sum_j F_{i,j} < 0\quad\text{or}\quad \sum_k F_{k,i} \geq \frac{c+\beta}{c-\gamma} \Delta_i. 
		\end{equation}
	\end{lemma}
	
	\begin{proof}
		Since $(z,F)\in \mathrm{Eq}(E_0,y_{\mathrm{null}})$ and $i$ is active, it means that $i$ has no interest in change to semi-stand-alone, either because $i$ in fact cannot pass to semi-stand-alone mode, or because $i$ can pass to semi-stand-alone but it is not profitable to do so.
		
		For $i$ to pass to semi-stand-alone, it will need to consider the operation $(z_i',F_i')$ given by
		\[
		F_{k,i}' = 0,\quad \forall k\in I\quad\text{ and }\quad z_i' = \frac{M_i}{C_{i,\Cout}}.
		\]
		With this in mind, we can analyse the reasons for $i$ not to go to semi-stand-alone.
		\begin{itemize}
			\item \textbf{semi-stand-alone is infeasible:} It is not hard to see that the semi-stand-alone operation $(z_i',F_i')$ is feasible for problem $P_i(z_{-i}, F_{-i},E)$ if and only if
			\[
			F_{i,0}(F_i',F_{-i}) = \frac{M_i}{C_{i,\Cout}} - \sum_{j} F_{i,j} \geq 0.
			\]
			Thus, if semi-stand-alone is infeasible, then \KC{$\frac{M_i}{C_{i,\Cout}} - \sum_{j} F_{i,j} < 0$}.
			
			\item  \textbf{semi-stand-alone is feasible but not profitable:} By equivalence \eqref{eq:EquivalenceDelta}, the nonprofitability of semi-stand-alone operation is equivalent to $\sum_{k} F_{k,i} \geq \frac{c+\beta}{c-\gamma}\Delta_i$.
		\end{itemize}
		By mixing both cases, we can write
		\begin{align*}
			i\text{ is active }&\iff \text{ semi-stand-alone is either infeasible or not profitable for }P_i(z_{-i},F_{-i},E)\\
			&\iff \frac{M_i}{C_{i,\Cout}} - \sum_{j} F_{i,j} < 0\quad\text{ or }\quad\sum_{k} F_{k,i} \geq \frac{c+\beta}{c-\gamma} \Delta_i.
		\end{align*}
		The proof is then finished.
	\end{proof}
	
	\subsection{Construction of the auxiliary optimization problem}
	
	Using the characterization of active agents in Lemma \ref{lemma:SecondActiveAgent}, we will develop an alternative single-level optimization problem  that, besides deciding the operation $(z,F)$, it ``decides'', for each agent $i\in I$, two more things:
	\begin{enumerate}
		\item It decides which of the two expressions of \eqref{eq:ExpressionF0-Piecewise} should be used to compute $F_{i,0}$.
		\item If $i$ is active, then it decides which of the two conditions in \eqref{eq:ConditionsActive-notSemiSTA} holds. 
	\end{enumerate}
	For this, we first introduce a boolean variable $y_{i,\act}\in \{0,1\}$ and define
	\begin{equation}\label{eq:newExpressionF0}
		\begin{aligned}
			F_{i,0}(y_{i,\act},F_{-i}) &= (1-y_{i,\act})\left( \frac{M_i}{C_{i,\Cout}}-\sum_{j} F_{i,j} \right)+y_{i,\act}\left(\frac{M_i}{C_{i,\Cout}-C_{i,\Cin}}-\sum_{j} F_{i,j}\right)\\
			&=\frac{M_i}{C_{i,\Cout}}(1-y_{i,\act})+\frac{M_i}{C_{i,\Cout}-C_{i,\Cin}}y_{i,\act}-\sum_{j} F_{i,j}
		\end{aligned}
	\end{equation}
	We will include the constraint that, whenever there exists $F_{k,i}>0$, then necessarily $y_{i,\act} = 1$. This is achieved by the constraint
	\begin{equation}\label{eq:ActiveConstraint}
		\sum_{k} F_{k,i} \leq Ky_{i,\act},
	\end{equation}
	for the constant $K>0$ given by \eqref{eq:ValueOfK}. Thus, if $i$ is in active situation, then $y_{i,\act} = 1$. Otherwise, we can choose the value of $y_{i,\act}$. We define $y_{\act} = (y_{i,\act}\ :\ i\in I)$. Observe that, since $\frac{M_i}{C_{i,\Cout}} \leq \frac{M_i}{C_{i,\Cout}-C_{i,\Cin}}$, it is always desirable to put $y_{i,\act} = 0$ whenever it is possible, when minimizing $F_0(F)$. Thus, the values of $y_{i,\act}$ will be consistent at any optimal configuration.
	
	Now, whenever $i$ is active, we need to decide which condition of \eqref{eq:ConditionsActive-notSemiSTA} will hold. We include a boolean variable $y_{i,\mathrm{pos}}\in \{0,1\}$, and the constraints
	\begin{align}
		&0 > \frac{M_i}{C_{i,\Cout}} - \sum_{j} F_{i,j} -  Ky_{i,\mathrm{pos}}\label{eq:FirstCaseActive}\\
		&0\leq \sum_{k} F_{k,i} - \frac{c+\beta}{c-\gamma} \Delta_i\cdot y_{i,\mathrm{pos}} +K(1-y_{i,\act})\label{eq:SecondCaseActive}
	\end{align}
	The above system works as follows. On the one hand, if $\frac{M_i}{C_{i,\Cout}} - \sum_{j} F_{i,j} \geq 0$, this forces $y_{i,\mathrm{pos}} = 1$. Then, if $y_{i,\act} = 0$, the constraint \eqref{eq:SecondCaseActive} becomes the profitability constraint: since semi-stand-alone is feasible, the active situation of agent $i$ must be better than semi-stand-alone. On the other hand, whenever  $\frac{M_i}{C_{i,\Cout}} - \sum_{j} F_{i,j} < 0$, we can choose $y_{i,\mathrm{pos}} = 0$ and thus constraint \eqref{eq:SecondCaseActive} is trivially verified by the positivity of fluxes. Finally, if $y_{i,\act} = 0$, then constraint \eqref{eq:SecondCaseActive} is trivially verified regardless the value of $y_{i,\mathrm{pos}}$.  We define $y_{\mathrm{pos}} = (y_{i,\pos}\ :\ i\in I)$.

	Let us denote $y= (y_{\mathrm{null}},y_{\act},y_{\mathrm{pos}})$ and by $y_i = (y_{i,\mathrm{null}}, y_{i,\act},y_{i,\pos})$. To construct the desired auxiliary problem, instead of considering $(z,F)\in \mathrm{Eq}(E,y_{\mathrm{null}})$, we will instead ask that the tuple $(y,z,F)$ to respect all the constraints of problems \eqref{eq:ProblemOfIndependentAgent_EIP} assuming $E = E^{\max}$, and to respect the new constraints \eqref{eq:FirstCaseActive} and \eqref{eq:SecondCaseActive}. To do so, we first define a new feasibility set
	\begin{equation}\label{eq:FeasibleEIP}
		\mathcal{S} = \left\{ (y,z,F)\ :\ \begin{array}{ll}
			\mbox{Equations
				\eqref{eq:watermass_balance}
				and \eqref{eq:Inletoutlet_independent} hold},\quad&\forall i\in I\\
			\sum_{j} F_{i,j} \leq K\cdot(1-y_{i,\mathrm{null}}),\qquad&\forall i\in I\\
			\sum_{k} F_{k,i} \leq K\cdot (1-y_{i,\mathrm{null}}),\qquad&\forall i\in I\\	 
			\sum_{k} F_{k,i} \leq Ky_{i,\act},\quad&\forall i\in I\\
			y_{i,\act},y_{i,\mathrm{pos}},y_{i,\mathrm{null}}\in \{0,1\},\quad&\forall i\in I\\
			z_{i}\geq 0,\quad&\forall i\in I\\
			F_i\geq 0,\quad&\forall i\in I\\
			F_{i,0}(y_{i,\act},F_{-i})\geq 0,\quad&\forall i\in I\\
			\sum_{k\in I} F_{k,i} - \frac{c+\beta}{c-\gamma} \Delta_i\cdot y_{i,\mathrm{pos}} +K(1-y_{i,\act})\geq 0,\quad&\forall i\in I
		\end{array} \right\}
	\end{equation}
	where $F_{i,0}(y_{i,\act},F_{-i})$ is given by equation \eqref{eq:newExpressionF0}. The first active condition given by \eqref{eq:FirstCaseActive} is left outside the set $\mathcal{S}$ since we will deal with it in the next section.
	
	With $\mathcal{S}$, the active condition \eqref{eq:FirstCaseActive} and the new cost constraint given by formula \eqref{eq:NewFormulaSTC}, we define the auxiliary Mixed-Integer optimization problem as
	\begin{equation}\label{eq:AuxiliaryProblem}
		\mathcal{A}_0 := \left\{\begin{array}{cl}
			\displaystyle\min_{y,z,F} &\displaystyle\sum_{i\in I} F_0(y_{i,\act},F_{-i})\\
			\\
			s.t&\begin{cases}
				(y,z,F)\in \mathcal{S},\\
				\frac{M_i}{C_{i,\Cout}} - \sum_{j} F_{i,j} -  Ky_{i,\mathrm{pos}}<0,\quad&\forall i \in I,\\
				\Cost_{i}(y_i,z_i,F_i,F_{-i}) \leq \KC{\STC_i(\alpha, y_{i,\mathrm{null}})},\quad&\forall i\in I.
			\end{cases}
		\end{array}\right.
	\end{equation}
	where $y_i = (y_{i,\mathrm{null}},y_{i,\act},y_{i,\mathrm{pos}})$, and the cost function $\Cost_{i}(y_i,z_i,F_i,F_{-i})$ is the same as in \eqref{eq:CostEIP} considering $E_0 = E_{\max}$, but with the expression \eqref{eq:newExpressionF0} for $F_0$. We finish this section by showing that problem $\mathcal{A}_0$ is in fact a relaxation of problem $\eqref{eq:ProblemOptDesign-NullClass}$.  
	
	\begin{lemma}\label{lemma:FeasiblePointsAreInAuxProblem} Let $(E,y_{\mathrm{null}},z,F)$ be a feasible point of Problem \eqref{eq:ProblemOptDesign-NullClass}, then there exists a pair $(y_{\act},y_{\mathrm{pos}})\in \{0,1\}^{2n}$ such that $(y,z,F)$ is feasible for $\mathcal{A}_0$.
	\end{lemma}
	\begin{proof}
		We construct $y_{\act}$ and $y_{\mathrm{pos}}$ as follow:
		\begin{enumerate}
			\item If $i$ is in stand-alone operation, then the economic constraint forces $y_{i,\mathrm{null}} = 1$. We set $y_{i,\act} = 0$ and $y_{i,\mathrm{pos}} = 1$. Then, the operation of $i$ respects the stand-alone operation and all the constraints for active agents are trivially fulfilled for $i$.
			\item If $i$ is in semi-stand-alone operation, we set again  $y_{i,\act} = 0$ and $y_{i,\mathrm{pos}} = 1$, and the same conclusion follows.
			\item Finally, if $i$ is active, we set $y_{i,\act} = 1$. Then,
			\begin{itemize}
				\item if $\frac{M_i}{C_{i,\Cout}} - \sum_{j} F_{i,j}<0$, we set $y_{i,\mathrm{pos}} = 0$.
				\item if $\frac{M_i}{C_{i,\Cout}} - \sum_{j} F_{i,j}\geq 0$, we set $y_{i,\mathrm{pos}} = 1$. Lemma \ref{lemma:SecondActiveAgent} ensures that in this case the profitability constraint \eqref{eq:SecondCaseActive} is verified.
			\end{itemize}
		\end{enumerate}
		Clearly with this construction $(y,z,F)$ is feasible for $\mathcal{A}_0$, and so the proof is then finished.
	\end{proof}
	
	
	\subsection{Approximated solutions and Main results}\label{subsec:ApproxSolutions}
	
	The main difficulty with the auxiliary problem $\mathcal{A}_0$, is that the feasibility set 
	\begin{equation}\label{eq:FeasibilitySetA0}
		S_0 = \left\{ (y,z,F) \ :\ \begin{array}{ll}(y,z,F)\in \mathcal{S}\\
			\frac{M_i}{C_{i,\Cout}} - \sum_{j} F_{i,j} -  Ky_{i,\mathrm{pos}}<0,\quad&\forall i \in I\\
			\Cost_{i}(y_{i},z_i,F_i,F_{-i}) \leq \KC{\STC_i(\alpha, y_{i,\mathrm{null}})},\quad&\forall i\in I
		\end{array}
		\right\}
	\end{equation} 
	might not be closed, due to the strict constraint \eqref{eq:FirstCaseActive}. To avoid this issue, we will consider, for each $\varepsilon>0$, an approximated version given by
	\begin{equation}\label{eq:ApproximatedAuxiliaryProblem}
		\mathcal{A}_{\varepsilon} := \left\{\begin{array}{cl}
			\displaystyle\min_{y,z,F} &\displaystyle\sum_{i\in I} F_0(y_{i,\act},F_{-i})\\
			\\
			s.t&\begin{cases} (y,z,F)\in \mathcal{S},\\
				\frac{M_i}{C_{i,\Cout}} - \sum_{j} F_{i,j} -  Ky_{i,\mathrm{pos}}\leq-\varepsilon,\quad&\forall i \in I,\\
				\Cost_{i}(y_{i},z_i,F_i,F_{-i}) \leq \KC{\STC_i(\alpha, y_{i,\mathrm{null}})},\quad&\forall i\in I.
			\end{cases}
		\end{array}\right.
	\end{equation}
	where the feasibility set $\mathcal{A}_{\varepsilon}$ is now given by
	\begin{equation}\label{eq:FeasibilitySetAEpsilon}
		S_{\varepsilon} = \left\{ (y,z,F) \ :\ \begin{array}{ll}(y,z,F)\in \mathcal{S}\\
			\frac{M_i}{C_{i,\Cout}} - \sum_{j} F_{i,j} -  Ky_{i,\mathrm{pos}}\leq-\varepsilon,\quad&\forall i \in I\\
			\Cost_{i}(y_{i},z_i,F_i,F_{-i}) \leq \KC{\STC_i(\alpha, y_{i,\mathrm{null}})},\quad&\forall i\in I
		\end{array}
		\right\}
	\end{equation}
	Clearly, $S_{\varepsilon}$ is a closed set and 
	\[
	S_0  = \bigcup_{\varepsilon>0}S_{\varepsilon}.
	\]
	The interpretation of $S_{\varepsilon}$ is that, whenever the semi-stand-alone operation is feasible for an active agent $i\in I$ with an error of $\varepsilon$, then the profitability condition \eqref{eq:SecondCaseActive} must hold with $y_{i,\mathrm{pos}} = 1$.
	
	\begin{thm}\label{thm:MainReformulationTheorem} For $\varepsilon>0$, let $(y,z,F)$ be a solution of problem $\mathcal{A}_{\varepsilon}$. Then, there \KC{exists} $E\in \mathcal{E}$ such that $(E,y_{\mathrm{null}},z,F)$ is feasible for Problem \eqref{eq:ProblemOptDesign-NullClass}.
	\end{thm}
	\begin{proof}
		
		We only need to show that there exists $E\in \mathcal{E}$ such that $(z,F)\in \mathrm{Eq}(E,y_{\mathrm{null}})$. To construct $E$, we start with $E = E^{\max}$, and then, for every $i\in I$, we do the following:
		\begin{enumerate}
			\item If $y_{i,\mathrm{null}} = 0$ (that is, $i$ is in stand-alone operation), we update
			\[
			E \leftarrow E \setminus\{(k,j)\in E^{\max}\ :\ k=i\text{ or } j=i  \}.
			\]
			\item If $\sum_k F_{k,i} = 0$, ( that is, $i$ is in stand-alone operation), we update
			\[
			E \leftarrow E \setminus\{(k,i) :\ k\in I\setminus\{ i\}  \}.
			\]
		\end{enumerate}
		We will show that this $E$ (which eliminates all connections for stand-alone agents, and all inlet connections for semi-stand-alone agents) verifies the inclusion $(z,F)\in \mathrm{Eq}(E,y_{\mathrm{null}})$. Reasoning by absurd, let us assume that there exists $i\in I$ and $(z_i',F_i')\neq (z_i,F_i)$ such that
		\begin{itemize}
			\item $(z_i',F_i')$ is feasible for $P_i(z_{-i},F_{-i},E)$, and
			\item $(z_i',F_i')$ is profitable for $i$, that is, $\Cost_i(z_i',F_i',F_{-i},E) < \Cost_i(z_i,F_i,F_{-i},E)$.
		\end{itemize}
		Now, if agent $i$ was either in stand-alone or semi-stand-alone situation, the construction of $E$ yields that there is \KC{no}  inlet connection between another enterprise and $i$. Thus, $F_{k,i} = 0 = F'_{k,i}$ for all $k\in I\setminus \{i\}$. Furthermore, the water balance constraint implies that 
		\[
		z_i = \frac{M_i}{C_{i,\Cout}} = z_i'.
		\]
		This would yield a contradiction with the fact that $(z'_i,F'_i)\neq (z_i,F_i)$. Thus, we have necessarily that $i$ is active, that is, $\sum_{k}F_{k,i} > 0$.
		
		Now, since $(z,F)\in S_{\varepsilon}$, then either
		\[
		\frac{M_i}{C_{i,\Cout}} - \sum_{j} F_{i,j} \leq -\varepsilon\quad\text{ or }\quad\sum_{k} F_{k,i} \geq \frac{c+\beta}{c-\gamma} \Delta_i.
		\]
		This means that either the semi-stand-alone operation is infeasible for $P_i(z_{-i},F_{-i},E)$ or it is simply not profitable for $i$. Thus, by the assumption of profitability of $(z_i',F_i')$, we conclude that  necessarily $\sum_k F_{k,i}' >0$. 
		
		We claim now that $(y,z_i',z_{-i},F_i',F_i)$ is also a feasible point of $S_{\varepsilon}$. To show this, it is enough to prove that
		\[
		\frac{M_i}{C_{i,\Cout}} - \sum_{j} F_{i,j} \leq -\varepsilon\quad\text{ or }\quad\sum_{k} F'_{k,i} \geq \frac{c+\beta}{c-\gamma} \Delta_i.
		\]
		If $\frac{M_i}{C_{i,\Cout}} - \sum_{j} F_{i,j} \leq -\varepsilon$ our claim verifies trivially, so let us assume that $\frac{M_i}{C_{i,\Cout}} - \sum_{j} F_{i,j} > -\varepsilon$.  If $\sum_{k} F'_{k,i} < \frac{c+\beta}{c-\gamma} \Delta_i$, this means, by Lemma \ref{lemma:FirstActiveAgent}, that $\Cost_i(z'_i,F_i',F_{-i},E) < \Cost_{i}(M_i/C_{i,\Cout},0,F_{-i},E)$, where the latter is the cost associated to the semi-stand-alone operation (regardless if it is feasible or not). Then, we can write
		\[
		\Cost_i(z'_i,F_i',F_{-i},E) < \Cost_{i}(M_i/C_{i,\Cout},0,F_{-i},E) \leq \Cost_i(z_i,F_i,F_{-i},E),
		\]
		where the second inequality follows from Lemma \ref{lemma:FirstActiveAgent} and the fact that $\sum_{k} F'_{k,i} \geq \frac{c+\beta}{c-\gamma} \Delta_i$. This contradicts the profitability assumption of $(z_i',F_i')$. Thus, 
		$\sum_{k} F'_{k,i} \geq \frac{c+\beta}{c-\gamma} \Delta_i$. Thus, it is not hard to see that keeping $y_{i,\act}$ and $y_{i,\pos}$ with the same values,  we get that $(y,z_i',z_{-i},F_i',F_i)$  verifies the constraints \eqref{eq:FirstCaseActive} and \eqref{eq:SecondCaseActive}. Since everything else is also verified, we conclude that $(y,z_i',z_{-i},F_i',F_i)$ is a feasible point of $S_{\varepsilon}$, as claimed.
		
		Now, since $(y,z,F)$ is an optimal solution of $\mathcal{A}_{\varepsilon}$, we get that
		\begin{align*}
			z_i' - z_i &= \sum_{i\in I} z_i - \left(z_i' + \sum_{k\neq i} z_i \right)=F_0(F) - F_0(F_i',F_{-i})\geq 0,
		\end{align*}
		where the second equality follows from the aggregated water balance constraint. Thus, noting that
		\[
		F_{i,0}(F) = \frac{M_i}{C_{i,\Cout} - C_{i,\Cin}} - \sum_j F_{i,j} = F_{i,0}(F_i',F_{-i}),
		\]
		we deduce, using the water balance constraint \eqref{eq:watermass_balance} for agent $i$, that
		\[
		z_i'	+ \sum_k F_{k,i}' =  F_{i,0}(F_i',F_{-i}) + \sum_{j}F_{i,j} =  F_{i,0}(F) + \sum_{j}F_{i,j}  = z_i	+ \sum_k F_{k,i}.
		\]
		This yields that $ \sum_k F_{k,i}' -  \sum_k F_{k,i} = z_i - z_i'$. Thus, we can write
		\begin{align*}
			&\Cost_{i}(z_i',F_i',F_{-i},E) - \Cost_{i}(z_i,F_i,F_{-i},E)\\
			=\,\,& c(z_i' - z_i) - \gamma\left(\sum_k F_{k,i}' -  \sum_k F_{k,i}\right) + \beta\left(  F_{i,0}(F_i',F_{-i}) -  F_{i,0}(F) \right)\\
			=\,\,& (c-\gamma)(z_i'-z_i)\geq 0.
		\end{align*}
		This is a contradiction with the profitability of $(z_i', F_i')$. Thus, such a deviation $(z_i',F_i')$ can not exists, and we conclude that $(z,F)\in \mathrm{Eq}(E,y_{\mathrm{null}})$. This finishes the proof.
	\end{proof}
	
	The above theorem links the optimal solutions of the problems $\mathcal{A}_{\varepsilon}$ with the designer's problem \eqref{eq:ProblemOptDesign-NullClass}. By mixing the above theorem with Lemma \ref{lemma:FeasiblePointsAreInAuxProblem}, the next theorem shows that in fact the solutions of $(\mathcal{A}_{\varepsilon}\ :\ \varepsilon>0)$ approximate the optimal value of \eqref{eq:ProblemOptDesign-NullClass}, which coincides with the optimal value of $\mathcal{A}_0$.
	
	\begin{thm}\label{thm:ConvergenceOfOptimalValues}
		Let us define
		\begin{align*}
			Z &:= \text{Optimal value of Problem  \eqref{eq:ProblemOptDesign-NullClass}}\\
			Z_0 &:=	\text{Optimal value of Problem }\mathcal{A}_0\\
			Z_\varepsilon &:=	\text{Optimal value of Problem }\mathcal{A}_\varepsilon
		\end{align*}
		We have that $Z = Z_0$, that the function $\varepsilon\in (0,+\infty)\to Z_{\varepsilon}$ is non-decreasing, and that
		\[
		Z_{\varepsilon}\xrightarrow{\varepsilon\to 0}Z_0.
		\]
	\end{thm}
	\begin{proof}
		Lemma \ref{lemma:FeasiblePointsAreInAuxProblem} ensures that $Z_0\leq Z$ since, after the transformation given in the proof of the lemma, every feasible point of Problem \eqref{eq:ProblemOptDesign-NullClass} is also a feasible point of $\mathcal{A}_0$. 
		
		Now, by construction,  we have that the family of feasibility sets $(S_{\varepsilon}\ :\ \varepsilon>0)$ is non-increasing with respect to $\varepsilon\in (0,+\infty)$: that is,
		\[
		\varepsilon' > \varepsilon \implies S_{\varepsilon'} \subseteq S_{\varepsilon}.
		\]  
		This yields that the function $\varepsilon\in (0,+\infty)\to Z_{\varepsilon}$ is non-decreasing, since  $Z_{\varepsilon} = \inf\{ F_0(F)\ :\ (y,z,F)\in S_{\varepsilon} \}$. Indeed, the infimum over a smaller set is bigger, and thus
		\[
		\varepsilon' > \varepsilon \implies  \KC{Z_{\varepsilon'} \geq Z_{\varepsilon}}.
		\]
		The convergence $Z_{\varepsilon}\xrightarrow{\varepsilon\to 0}Z_0$ follows immediately from the fact that $S_0 = \bigcup_{\varepsilon>0} S_{\varepsilon}$. 
		
		Finally, let $\varepsilon>0$. Since Theorem \ref{thm:MainReformulationTheorem} ensures that every solution $(y^*,z^*,F^*)$ of $\mathcal{A}_{\varepsilon}$ induces a feasible point of Problem \eqref{eq:ProblemOptDesign-NullClass} with the same value for the objective function, we get that
		\[
		Z \leq F_0(F^*) = Z_{\varepsilon}.
		\]
		Thus, $Z_0\leq Z\leq Z_{\varepsilon}$ for every $\varepsilon>0$, which implies that $Z= Z_0$. The proof is finished.
	\end{proof}

	\DS{
		The last corollary ensures that we can find the solution of Problem \eqref{eq:ProblemOptDesign-NullClass}, whenever it exists, by solving $\mathcal{A}_{\varepsilon}$ for $\varepsilon>0$ small enough.
		
		\begin{cor}\label{cor:ExactSolution}
			Problem \eqref{eq:ProblemOptDesign-NullClass} admits a solution if and only if there exists $\varepsilon^*>0$ such that
			\[
			Z_{\varepsilon} = Z_0 = Z_{\varepsilon^*},\quad \forall \varepsilon\in (0,\varepsilon^*].
			\]
			In such a case, any solution $(y^*,z^*,F^*)$ of $\mathcal{A}_{\varepsilon^*}$ is a solution of $\mathcal{A}_0$, and it is also a solution of Problem \eqref{eq:ProblemOptDesign-NullClass}: that is, there exists $E^*\in \mathcal{E}$ such that $(E^*,y_{\mathrm{null}}^*,z^*,F^*)$ is a solution of Problem \eqref{eq:ProblemOptDesign-NullClass}.
		\end{cor}
		
		\begin{proof}
			Let us show first the necessity. Let $(E^*,y_{\mathrm{null}}^*,z^*,F^*)$ be a solution of Problem \eqref{eq:ProblemOptDesign-NullClass}, and let $(y^*,z^*,F^*)$ its associated feasible point in $S_0$. Since $S_0 = \bigcup_{\varepsilon>0}S_{\varepsilon}$ we get that there exists $\varepsilon^*>0$ such that $(y^*,z^*,F^*)\in \mathcal{A}_{\varepsilon^*}$. Furthermore, the family  $(S_{\varepsilon}\ :\ \varepsilon>0)$ is non-increasing with respect to $\varepsilon\in (0,+\infty)$, we get that $(y^*,z^*,F^*)\in S_{\varepsilon}$ for every $\varepsilon>0$.
			
			Since $(E^*,y_{\mathrm{null}}^*,z^*,F^*)$ is optimal, we get that $F_0(F^*) = Z = Z_0$. Thus,  $(y^*,z^*,F^*)$ is also optimal for $\mathcal{A}_0$. Finally, for every $\varepsilon\in (0,\varepsilon^*]$, we get that
			\[
			(y^*,z^*,F^*)\in S_{\varepsilon}\implies Z_{\varepsilon}\leq F_0(F^*) = Z_0\leq Z_{\varepsilon}.
			\]
			Thus, $Z_{\varepsilon} = Z_0 = Z_{\varepsilon^*}$ for every $\varepsilon\in (0,\varepsilon^*]$. The rest of the proof follows directly form Theorem \ref{thm:MainReformulationTheorem}.
			
			Now, for the sufficiency, let $\varepsilon^*>0$ such that $Z_0 = Z_{\varepsilon^*}$, and let $(y^*,z^*,F^*)$ be a solution of $\mathcal{A}_{\varepsilon^*}$. Then, since $(y^*,z^*,F^*)\in S_{\varepsilon^*}\subset S_0$, we deduce that $(y^*,z^*,F^*)$ is a solution of $A_{0}$. Let $E^* \in \mathcal{E}$ given by
			\[
			E^* = \{ (i,j)\ :\ i,j\in I,\, i\neq j,\, F_{i,j}>0  \}\cup E^{st}.
			\]
			Then, it is not hard to see that $(E^*,y_{\mathrm{null}}^*,z^*,F^*)$ is feasible for Problem \eqref{eq:ProblemOptDesign-NullClass}, and Lemma \ref{lemma:FeasiblePointsAreInAuxProblem} entails that is is also optimal. The proof is then completed.
		\end{proof}
		
	}
	
	
	\subsection{Numerical Implementation}\label{sec:Num_Imp}
	Based on the developments of this section, our aim is to solve approximately $\mathcal{A}_0$ by solving $\mathcal{A}_{\varepsilon}$ for $\varepsilon>0$ small enough. However, this strategy has two drawbacks: first, we don't really know how small $\varepsilon$ should be, and second, we are limited by the numerical precision of our computer. To try to measure how far a solution of $\mathcal{A}_\varepsilon$ is from a solution $\mathcal{A}_{0}$, we will use an optimality gap. Consider the problem 
	
	\begin{equation}\label{eq:GapAuxiliaryProblem}
		\overline{\mathcal{A}} := \left\{\begin{array}{cl}
			\displaystyle\min_{y,z,F} &\displaystyle\sum_{i\in I} F_0(y_{i,\act},F_{-i})\\
			\\
			s.t&\begin{cases}
				(y,z,F)\in \mathcal{S},\\
				\frac{M_i}{C_{i,\Cout}} - \sum_{j} F_{i,j} -  Ky_{i,\mathrm{pos}}\leq 0,\quad&\forall i \in I,\\
				\Cost_{i}(y_i,z_i,F_i,F_{-i}) \leq \KC{\STC_i(\alpha, y_{i,\mathrm{null}})},\quad&\forall i\in I.
			\end{cases}
		\end{array}\right.
	\end{equation}
	
	Since the feasible set of $\overline{\mathcal{A}}$ is bigger than $S_0$, we get that the optimal value of $\overline{\mathcal{A}}$, which we denote $\overline{Z}$, must be smaller than $Z_0$. That is, we have that
	\begin{equation}\label{eq:GapInequality}
		\overline{Z} \leq Z_0 \leq Z_{\varepsilon},\quad\forall \varepsilon>0.
	\end{equation}
	Unfortunately, the inequality $\overline{Z}\leq \overline{Z}_0$ might be strict. This is due to the fact that the feasible set of $\overline{\mathcal{A}}$ is not necessarily the closure of the feasible set of $\mathcal{A}_0$. However, $\overline{Z}$ gives us a reference point. Based on this, we propose the following methodology to solve problem $\mathcal{A}_0$:
	\begin{enumerate}
		\item 	Solve $\overline{\mathcal{A}}$, finding $(\bar{y},\bar{z},\bar{F})$ and $\overline{Z}$. If $(\bar{y},\bar{z},\bar{F})$ is feasible for $\mathcal{A}_0$, then this is the solution.
		\item  If not, then choose $\varepsilon>0$ as small as possible within the numerical precision. 
		\item Solve $\mathcal{A}_{\varepsilon}$, finding $(y_{\varepsilon},z_{\varepsilon}, F_{\varepsilon})$ and $Z_{\varepsilon}$ and keep  $(y_{\varepsilon},z_{\varepsilon}, F_{\varepsilon})$ as the best possible solution. 
	\end{enumerate}
	We can compare $Z_{\varepsilon}$ with $\overline{Z}$ to get an optimality gap: \DS{if $Z_{\varepsilon} - \bar{Z} = 0$, it means that $Z_{\varepsilon} = Z_0$ and thus the solution $(y_{\varepsilon},z_{\varepsilon}, F_{\varepsilon})$ is in fact a solution of $\mathcal{A}_0$, in view of Corollary \ref{cor:ExactSolution}.} If $Z_{\varepsilon}-\overline{Z}$ is not zero but it is small, we know that $Z_{\varepsilon}$ is close to $Z_0$ and so $(y_{\varepsilon},z_{\varepsilon}, F_{\varepsilon})$ is close to be a solution of $\mathcal{A}_0$.  However, if $Z_{\varepsilon}-\overline{Z}$ is ``large'', we can not get any conclusion. In such a case, we only can say that $(y_{\varepsilon},z_{\varepsilon}, F_{\varepsilon})$ is the best solution we could get.
	
	
	\section{Numerical experiments \label{sec: numerical results}}
	In this section, we present numerical examples of the Control-Input model. The objective of this model is to design and optimize the water exchange networks in the EIP to minimize the total freshwater consumption (equivalent to the total discharge of polluted water) of the enterprises and minimize the operating cost of enterprises, by considering that there is an exchange of water between enterprises and that each enterprise controls his input flux. The detailed model we use is described in Section \ref{sec:EIP}.  The optimization problem we solve is problem $\mathcal{A}_{0}$. Furthermore, to solve $\mathcal{A}_{0}$ we follow the methodology that presented in Sub-section \ref{sec:Num_Imp}.

	\subsection{Results and discussion of case study}\label{subsec:case_study}

	The study case we present consists of an EIP made up of 15 enterprises, each one including only one process and one type of contaminant. Data is partially inspired from \cite{Ramos2016, OlesenPolley1996}.  The data of 15 enterprises is given in Table \ref{Parameters}. Prices are shown in Table \ref{Associated costs}.  It is assumed that the EIP operates for one hour, that is, A = 1 h.

	\begin{table}[H]
		\centering
		\begin{tabular}{|c|c|c|c|}
			\hline 
			\mbox{Enterprise} $i$ &  $\mathrm{C}_{i,\mathrm{in}}$(ppm) & $\mathrm{C}_{i,\mathrm{out}}$(ppm) &$\mathrm{M}_{i}$(g/h) \\ 
			\hline  \hline
			1 & 30 & 100 & 7500 \\ 
			\hline 
			2& 0 &  200 & 6000 \\ 
			\hline 
			3& 50 & 100 & 5000 \\ 
			\hline
			4& 80 & 800 & 30000 \\ 
			\hline
			5& 400 & 800 & 4000 \\ 
			\hline
			6 & 20 & 100 & 2500 \\ 
			\hline 
			7& 50 &  100 & 2200 \\ 
			\hline 
			8& 80 & 400 & 5000 \\ 
			\hline
			9& 100 & 800 & 30000 \\ 
			\hline
			10& 400 & 1000 & 4000 \\ 
			\hline
			11 & 30 & 60 & 2000 \\ 
			\hline 
			12& 25 &  50 & 2000 \\ 
			\hline 
			13& 25 & 75 & 5000 \\ 
			\hline
			14& 50 & 800 & 30000 \\ 
			\hline
			15& 100 & 900 & 13000 \\ 
			\hline 
		\end{tabular}
		\caption{Enterprises' Parameters.}
		\label{Parameters} 
	\end{table}
	
	\begin{table}[H]
		\centering
		\begin{tabular}{cc}  
			\hline 
			Parameter  & Value ($\$$/tonne) \\ 
			\hline 
			$c$ & 0.13 \\ 
			$\beta$ & 0.22 \\ 
			$\gamma$ & 0.01\\ 
			\hline
		\end{tabular}
		\caption{Associated costs.} 
		\label{Associated costs} 
	\end{table}
	
	All simulations where implemented with \texttt{Julia v1.0.5} programming language \cite{bezanson2017julia}, using \texttt{Cplex} as solver.
	
	Now, we present the results of the simulations with the data given in Table \ref{Parameters} and Table \ref{Associated costs}. Our aim is to determine the optimal networks for the water exchange network in the EIP so that the total freshwater consumption (equivalent to the total discharge of polluted water) of the enterprises in the EIP is minimized and the participating enterprises in the EIP reduce at least 5\%. Therefore, we must solve optimization problem $\mathcal{A}_0$. To do so, 
	\begin{enumerate}
		\item First, we must solve problem $\overline{\mathcal{A}}$ to find $(\bar{y},\bar{z},\bar{F})$ and obtain $\overline{Z} = 332.46$ (T/h).
		\item Secondly, we verify whether  $(\bar{y},\bar{z},\bar{F})$ is feasible for problem $\mathcal{A}_0$ or not, if feasible $(\bar{y},\bar{z},\bar{F})$ will be an optimal solution for $\mathcal{A}_0$. Unfortunately, in this case $(\bar{y},\bar{z},\bar{F})$ is not feasible for $\mathcal{A}_0$.
		\item Finally, we will choose  $\varepsilon>0$ as small as possible, that is $\varepsilon = 10^{-6}$. Then we solve the problem $\mathcal{A}_{\varepsilon}$ to find $(y_{\varepsilon},z_{\varepsilon}, F_{\varepsilon})$ and obtain $Z_\varepsilon = 332.46$ (T/h).
	\end{enumerate}
	In this case study, we have $Z_\varepsilon = 332.46 = \overline{Z}$. Therefore, the solution $(y_{\varepsilon},z_{\varepsilon}, F_{\varepsilon})$ is exactly a solution of $\mathcal{A}_0$, according to Corollary \ref{cor:ExactSolution}.
	
	The optimal solution $(y_{\varepsilon},z_{\varepsilon}, F_{\varepsilon})$ provides an optimized configuration as shown in Figure \ref{figg:optimal_configuration}. This optimal network provides operating cost of each enterprise and total freshwater consumption that are lower than a stand-alone network as shown in Table \ref{table:ResultsEIP}.  The detailed results of fluxes in the optimal network are presented in Table \ref{table:values_of_fluex}.
	
	\begin{figure}[H]
		\fontsize{6pt}{6pt}\selectfont
		\centering
		\begin{tikzpicture}[->,>=stealth',shorten >=0pt,auto,node distance=1cm,
			semithick,square/.style={regular polygon,regular polygon sides=4}]
			
			\tikzstyle{every state}=[text=black,,font = \bfseries, fill opacity = 0.35, text opacity = 1]
			\node[state] 		 (A) [fill=gray]                   					{12};
			\node[state] 		 (B) [fill=gray,right of=A]       					{11};
			\node[state] 		 (C) [fill=gray,right of=B]       					{13};
			\node[state] 		 (D) [fill=gray,below of=B,node distance=0.6in]       					{3};
			\node[state]         (E) [left of=D] 						{9};	
			\node[state]         (F) [left of=E] 	{7};
			\node[state]         (G) [fill=gray,left of=F] 						{14};
			\node[state]         (H) [fill=gray,left of=G] 									{1};
			\node[state]         (I) [right of=D] 									{15};
			\node[state]         (J) [right of=I] 			{4};
			\node[state]         (K) [fill=gray,right of=J] 			{6};
			\node[state] 		 (L) [fill=gray, right of=K]     {8};
			\node[state]         (M) [pattern=north west lines, pattern color = gray,  right of=L] 	{2};
			\node[state]         (N) [pattern=north west lines, pattern color = gray,  right of=M] 			{5};
			\node[state]         (O) [pattern=north west lines, pattern color = gray,  left of=H,node distance=0.5in] 	{10};
			\node[state]         (P) [square,below of=D,node distance=0.6in]	{$0$};     
			
			\path (A) edge node{} (D) 
			(A) edge node{} (F)
			(B) edge node{} (D)
			(C) edge[bend right] node{} (A)
			(C) edge node{} (D)
			(C) edge node{} (J)
			(D) edge node {} (E)
			(F) edge node{} (E)	
			(F) edge[bend right] node{} (D)		
			(H) edge node{} (G)
			(D) edge node{} (I)
			(D) edge[bend right] node{} (J)
			(K) edge node{} (J)
			(K) edge node{} (L)
			(G) edge  node{} (P)
			(F) edge node{} (P)
			(E) edge node{} (P)
			(D) edge node{} (P)
			(I) edge node{} (P)
			(J) edge node{} (P)
			(K) edge node{} (P)
			(H) edge  [bend right=20] node{} (P)
			(L) edge [bend left=20] node{} (P)
			(M) edge [bend left=20] node{} (P)
			(N) edge [bend left=20] node{} (P)
			(O) edge [bend right=20] node{} (P);
		\end{tikzpicture}
		\caption{The configuration in the case $\alpha = 0.95$ and $\varepsilon = 10^{-6}$. Gray nodes are consuming strictly positive freshwater. Dashed nodes are operating in stand-alone mode.}
		\label{figg:optimal_configuration}
	\end{figure}
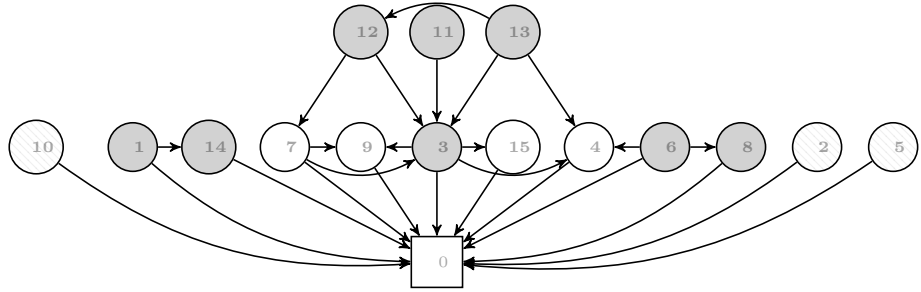
	
	\begin{table}[H] 
		\centering 
		\begin{tabular}{|c|c|c||c|c|c|}
			\hline
			Enterprise& Freshwater  & Freshwater & $\Cost_{i}$ Stand-Alone & $\Cost_{i}$ in EIP & \% Reduction\\
			& Stand-Alone (T/h) & in EIP (T/h)& (MMUSD/hour) & (MMUSD/hour) & in $\Cost_{i}$ \\ \hline
			1&75.00&75.00&26.25&22.05 & 16.00\\ \hline 
			2*&30.00&30.00&10.50&10.50 & 0.00 \\ \hline
			3&50.00&17.00&17.50&14.54 &  16.91\\ \hline
			4&37.50&0.00&13.13&9.58 & 26.98 \\ \hline
			5*&5.00&5.00&1.75&1.75& 0.00 \\ \hline
			6&25.00&25.00&8.75&5.61 & 35.86 \\ \hline
			7&22.00&0.00&7.70&5.50 & 28.57 \\ \hline
			8&12.50&3.13&4.38&3.97& 9.29 \\ \hline
			9&37.50&0.00&13.13&9.86& 24.90 \\ \hline
			10*&4.00&4.00&1.40&1.40&0.00  \\ \hline
			11&33.33&33.33&11.67&4.67& 60.00 \\ \hline
			12&40.00&53.33&14.00&8.00& 42.86 \\ \hline
			13&66.67&66.67&23.33&9.33& 60.00 \\ \hline
			14&37.50&20.00&13.13&11.60& 11.62 \\ \hline
			15&14.44&0.00&5.06&3.74& 26.07 \\ \hline
			\hline
			Total& 490.44	&332.46&	171.66	&122.10& 28.87\\ \hline		
		\end{tabular} 
		\caption{Summary of results of the EIP. Marked enterprises (*) are left outside the park, operating stand-alone.}
		\label{table:ResultsEIP}
	\end{table} 
	For this case study, when the enterprises operate stand alone, then the entire system consumes a total 490.44 (T/h) of freshwater and, of course, generates 490.44 (T/h) of wastewater. The optimal design obtain with SLMF approach allows to reduce its freshwater requirement to 332.46 (T/h), and thus generates 332.46 (T/h) of wastewater, which is equivalent to a reduction of 32.21 \%. Furthermore, the water demand of enterprises 4, 7, 9, and 15 are entirely supplied by other enterprises.  
	
	In addition, each participating enterprise has a cost reduction of at least 5\%,  as promised by the contract. Indeed, enterprises 11 and 13 have the highest reduction in operating costs up to 60\% compared to stand-alone configuration while enterprise 8 has the lowest operating cost reduction of only 9.29\% compared to stand-alone configuration. On the other hand, the EIP authority cannot guarantee that the enterprises 2, 5, and 10 reduce their operating costs by at least 5\%, thus these enterprises work stand-alone. Total operating cost is decreased regarding the stand-alone case, as expected, from 171.66 (\$/h) to 122.86 (\$/h), which means a decrease of 28.87\%.
	
	\subsection{Sensitivity analysis for parameter $\varepsilon$}\label{subsec:varepsilon}
	In this part, we present an example to show that the optimal value of Problem $\mathcal{A}_{\varepsilon}$ and the optimal value of Problem $\overline{\mathcal{A}}$ can be different, i.e., there is a gap between $Z_{\varepsilon}$ and $\overline{Z}$. To do so, we consider an EIP consisting of 10 enterprises presented in Table \ref{Parameters_Sensivity}. The prices are the same as before, presented in Table \ref{Associated costs}, and we fix the parameter $\alpha = 0.95$.
	
	\begin{table}[H]
		\centering
		\begin{tabular}{|c|c|c|c|}
			\hline 
			\mbox{Enterprise} $i$ &  $\mathrm{C}_{i,\mathrm{in}}$(ppm) & $\mathrm{C}_{i,\mathrm{out}}$(ppm) &$\mathrm{M}_{i}$(g/h) \\ 
			\hline  \hline
			1 & 50 & 100 & 5000 \\ 
			\hline 
			2& 80 &  800 & 30000 \\ 
			\hline 
			3& 400 & 800 & 4000 \\ 
			\hline
			4& 20 & 100 & 2500 \\ 
			\hline
			5& 80 & 400 & 5000 \\ 
			\hline
			6 & 100 & 800 & 30000 \\ 
			\hline 
			7& 30 &  60 & 2000 \\ 
			\hline 
			8& 25 & 50 & 2000 \\ 
			\hline
			9& 25 & 75 & 5000 \\ 
			\hline
			10& 50 & 800 & 30000 \\  
			\hline 
		\end{tabular}
		\caption{Enterprises' Parameters.}
		\label{Parameters_Sensivity} 
	\end{table}
	With data of 10 enterprises is presented in Table \ref{Parameters_Sensivity}, our aim is to solve the Problem $\overline{\mathcal{A}}$ and $\mathcal{A}_{\varepsilon}$. Indeed,
	\begin{itemize}
		\item  First, we solve problem $\overline{\mathcal{A}}$ to find $(\bar{y},\bar{z},\bar{F})$ and obtain $\overline{Z} = 201.46$ (T/h). 
		\item Then, we will choose  $\varepsilon>0$ as small as possible, that is $\varepsilon = 10^{-6}$. Then we solve the problem $\mathcal{A}_{\varepsilon}$ to find $(y_{\varepsilon},z_{\varepsilon}, F_{\varepsilon})$ and obtain $Z_\varepsilon = 201.48$ (T/h).
	\end{itemize} 
	In this case study, we have $Z_\varepsilon - \overline{Z} = 0.02$ (T/h). Thus, the gap between $Z_\varepsilon$ and $\overline{Z}$ is small, then we conclude that $(y_{\varepsilon},z_{\varepsilon}, F_{\varepsilon})$ is an approximate solution of $\mathcal{A}_0$. 
	\subsection{Comparison between the control input model and the blind input model}\label{subsec:CI_vs_BI}
	One of the goals of the control input model is to propose an alternative approach to  the blind input model developed in \cite{SALAS2020107053}. It is therefore important to compare the optimal results of both approaches. The criteria we used for comparing the two models are as follows:  first,  the ability to reduce freshwater consumption and thus wastewater; second, the number of enterprises involved into the optimal EIP and last, the ability to reduce the operating costs of each enterprise. These criteria will be discussed step by step.

	First of all and for both models, the total of freshwater consumption and wastewater discharge decreases when $\alpha$ increases as shown in Figure \ref{fig:Sensitivitydischarge_BIandCI}. Note that only the wastewater discharge curve is included below because, of course, both are the same. But the control input model allows an strong reduction of freshwater consumption and wastewater discharge compare to the blind input model. One of the reasons why is that, as it will be seen in Figure \ref{fig:SensitivitySTA_BIandCI} for the same relative cost reduction $\alpha$ there are always more enterprises participating in the EIP for control input model than with the blind input model.  In the optimized networks, the control input model achieves a minimum total of freshwater consumption as well as wastewater discharge is 332.46 (T/h) corresponding to 32.21\% with respect to the stand-alone configuration, while the blind input model achieves a minimum total freshwater consumption as well as wastewater discharge is 340.13 (T/h) corresponding to 30.65\% with respect to the stand-alone configuration. 
	
	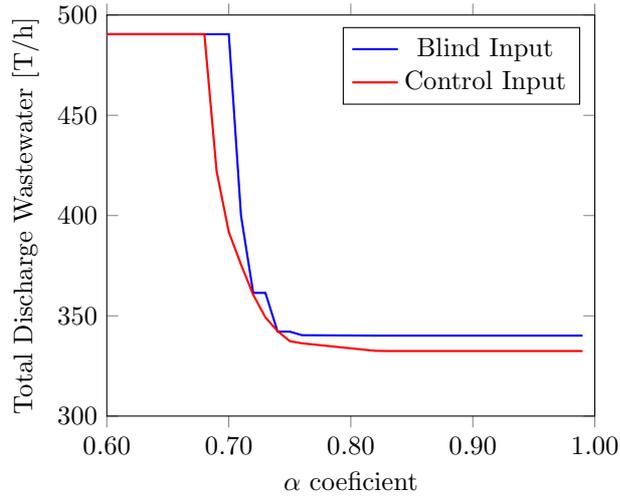
\begin{figure}[H]
		\centering
		\begin{tikzpicture}
			\pgfplotsset{
				xmin=0.6, xmax=1, legend pos = north east
			}
			
			\begin{axis}[
				ymin=300, ymax=500,
				xlabel=$\alpha$ coeficient,
				ylabel={Total Discharge Wastewater [T/h]} ,
				xtick align = outside,
				x tick label style={
					/pgf/number format/.cd,
					fixed,
					fixed zerofill,
					precision=2,
					/tikz/.cd
				}
				]
				\addplot[color=blue, thick] table [x index=0, y index=1, col sep=space]  {Plot_Latex_alpha_Fw_STA_Cost_BI.txt};
				\addplot[color=red, thick] table [x index=0, y index=1, col sep=space]  {Plot_Latex_alpha_Fw_STA_Cost_CI.txt};  
				\label{plot_discharge_BIandCI} 
				\addlegendentry{Blind Input}
				\addlegendentry{Control Input} 
			\end{axis}
			
		\end{tikzpicture}
		\caption{Sensitivity analysis for $\alpha\in [0.60,0.99]$ and $\varepsilon = 10^{-6}$. Total discharge wastewater.}\label{fig:Sensitivitydischarge_BIandCI}		
	\end{figure}

	The number of enterprises which operate stand-alone for both models is shown in Figure  \ref{fig:SensitivitySTA_BIandCI}. Roughly speaking, for $\alpha \in [0.60, 0.99]$, the number of enterprises operating stand-alone with control inputs is almost less than that of enterprises with blind inputs, except for a few special cases. Moreover, with $\alpha \in [0.69, 0.70]$, the control input model shows its clear superiority on the blind input model. Indeed with the control inputs the designer can build a park for which not only a wastewater discharge reduction is achieved compare to blind input results but also the designer can attract the exigent enterprises by guaranteeing a relative gain of more than 30\% while with such a rate blind input model only proposes full stand alone situation. Finally, for $\alpha \leq 0.68$, the optimal solution is the stand-alone configuration for both models, thus $0.68$ playing the role of a  threshold value for the relative gain.
	
	\begin{minipage}[c]{.5\linewidth}
		\begin{center}
			\begin{figure}[H]
				\begin{tikzpicture}
					\pgfplotsset{
						xmin=0.6, xmax=1, legend pos =  north east
					}
					
					\begin{axis}[
						ymin=0, ymax=16,
						xlabel=$\alpha$ coeficient,
						ylabel={Number of Stant-alone agents} ,
						xtick align = outside,
						x tick label style={
							/pgf/number format/.cd,
							fixed,
							fixed zerofill,
							precision=2,
							/tikz/.cd
						}
						]
						\addplot[color=blue, thick] table [x index=0, y index=2, col sep=space]  {Plot_Latex_alpha_Fw_STA_Cost_BI.txt};
						\addplot[color=red, thick] table [x index=0, y index=2, col sep=space]  {Plot_Latex_alpha_Fw_STA_Cost_CI.txt};
						\label{plot_STA_BIandCI}
						\addlegendentry{Blind Input}
						\addlegendentry{Control Input} 
					\end{axis}
					
				\end{tikzpicture}
				\caption{Sensitivity analysis for $\alpha\in [0.60,0.99]$ and $\varepsilon = 10^{-6}$. Number of stand-alone enterprises in the park.}\label{fig:SensitivitySTA_BIandCI}
			\end{figure}
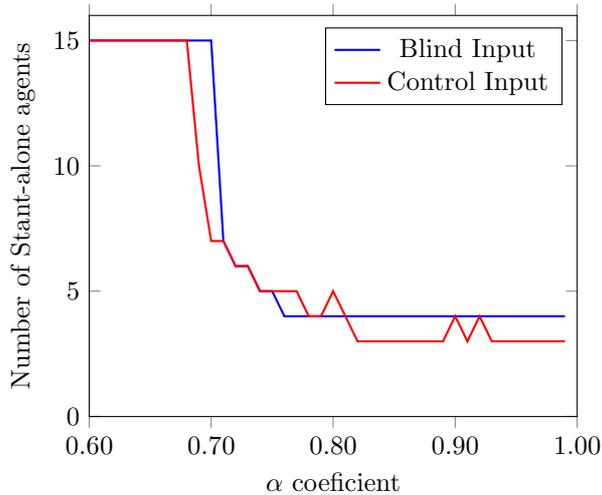 
		\end{center}
	\end{minipage} 
	$~~~~~$
	\begin{minipage}[c]{.5\linewidth}
		\begin{center}
			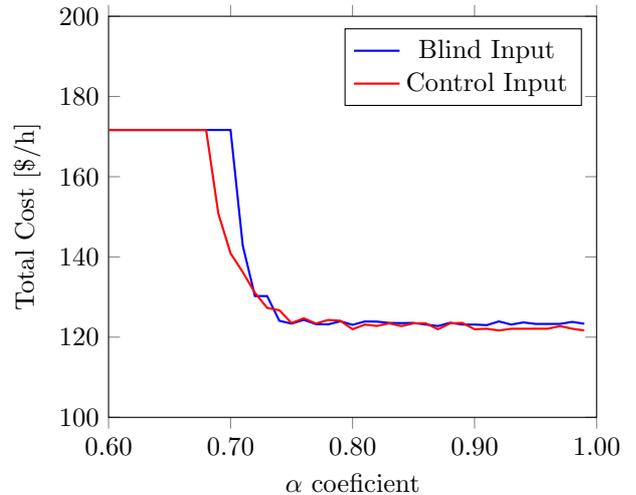
\begin{figure}[H]
				\begin{tikzpicture}
					\pgfplotsset{
						xmin=0.6, xmax=1, legend pos = north east
					}
					
					\begin{axis}[
						ymin=100, ymax=200,
						xlabel=$\alpha$ coeficient,
						ylabel={Total Cost [\$/h]} ,
						xtick align = outside,
						x tick label style={
							/pgf/number format/.cd,
							fixed,
							fixed zerofill,
							precision=2,
							/tikz/.cd
						}
						]
						\addplot[color=blue, thick] table [x index=0, y index=3, col sep=space]  {Plot_Latex_alpha_Fw_STA_Cost_BI.txt};
						\addplot[color=red, thick] table [x index=0, y index=3, col sep=space]  {Plot_Latex_alpha_Fw_STA_Cost_CI.txt};
						\label{plot_Cost_BIandCI}
						\addlegendentry{Blind Input}
						\addlegendentry{Control Input} 
					\end{axis}
					
				\end{tikzpicture}
				\caption{Sensitivity analysis for $\alpha\in [0.60,0.99]$ and $\Coef = 1$. Total operating cost in the park.}\label{fig:Sensitivity_Cost_BIandCI}
			\end{figure}
		\end{center}
	\end{minipage}
	As observed from Figure \ref{fig:Sensitivity_Cost_BIandCI}, the total operating cost in the EIP with control inputs is always less than the one with blind inputs. The behavior of the total cost with both models is declining evenly. In the optimized networks, the control input model achieves a minimum total operating cost is 121.66 (\$/h) corresponding to 29.13\% with respect to the stand-alone configuration, while the blind input model achieves a minimum total operating cost is 122.77 (\$/h), which is equivalent to a reduction of 28.48 \%..

	\section{Conclusion and perspectives}\label{sec:conclusion}
	
	In this work, we address the design and optimization of water exchange networks in eco-industrial parks 440 by formulating and solving Single-Leader-Multi-Follower games. Using some linear structure of the costs functions $\mathrm{Cost}_i(\cdot)$ of each enterprise, the control-input model can be solved by finding the solution of an auxiliary single-level mixed-integer problem. This auxiliary problem can be approximated by a family of usual single-level mixed-integer linear problems, which can be treated by common commercial solvers, thus allowing to tackle larger EIP networks. In our model, we consider that each participating enterprise 445 can control the amount of water coming from other enterprises. In another word, when participating in the network, each enterprise can control his input flux, which is more realistic than other models in the literature \cite{SALAS2020107053,Ramos2016}. The results show that the proposed approach is efficient. Indeed, the total discharge wastewater (=the total freshwater consumption) and total network cost in the optimal network have been reduced by 32.21 \% and 28.87\% respectively in which the designer ensures that each participating
	450 enterprise has a cost reduction of at least 5\%.
	Nevertheless our approach use the assumptions that each enterprise only manages one industrial process and that no regeneration units are used. This represents a limitation of our work, in particular concerning the second hypothesis since as it has been emphasized in \cite{SALAS2020107053} and \cite{Ramos2016}, the introduction of regeneration units usually allows some significant improvements of the results. It would be thus interesting, but out of the
	455 scope of this work, to analyse to which extend it could be possible to use our calculus and reformulation for an EIP in which either enterprises handle more than one processes or regeneration units are involved.

	\section*{Acknowledgments}
	The first and third authors is funded by MathAMSud project SOCCAE ``Stochastic Optimization and Chance Constraints with Applications to Energy" 2020-2021. The second author would like to thank the financial support provided by the ``R\'egion Occitanie" and the European Union (through the Fonds europ\'een de d\'eveloppement r\'egional (FEDER)). The third author was partially funded by Centro de Modelamiento Matem\'atico (CMM), FB210005, BASAL funds for centers of excellence from ANID-Chile.

	\appendix
	\section{Network Fluxes of Simulations} 
	The values of the flux corresponding to the case in sections \ref{sec: numerical results} is given in Tables \ref{table:values_of_fluex}. The entrance $(i,j)$ corresponds to the flux sent from enterprise $i$ to enterprise $j$.
	\begin{table}[hb] 
		\centering
		\resizebox{\textwidth}{!}{
			\begin{tabular}{|c|c|c|c|c|c|c|c|c|c|c|c|c|c|c|c|c|}
				\hline 
				Enterprise& 1 & 2 & 3 & 4 & 5 & 6 & 7 & 8 & 9 & 10 & 11 & 12 & 13 & 14 &15 &Sink  \\ 
				\hline 
				1& 0.00 & 0.00 & 0.00 & 0.00 & 0.00 & 0.00 & 0.00 & 0.00 & 0.00 & 0.00 & 0.00 & 0.00 & 0.00 & 20.00 & 0.00 & 55.00 \\ 
				\hline 
				2& 0.00 & 0.00 & 0.00 & 0.00 & 0.00 & 0.00 & 0.00 & 0.00 & 0.00 & 0.00 & 0.00 & 0.00 & 0.00 & 0.00 & 0.00 & 30.00 \\ 
				\hline 
				3& 0.00 & 0.00 & 0.00 & 5.89 & 0.00 & 0.00 & 0.00 & 0.00 & 27.86 & 0.00 & 0.00 & 0.00 & 0.00 & 0.00 & 16.25 & 50.00 \\ 
				\hline 
				4& 0.00 & 0.00 & 0.00 & 0.00 & 0.00 & 0.00 & 0.00 & 0.00 & 0.00 & 0.00 & 0.00 & 0.00 & 0.00 & 0.00 & 0.00 &  41.67\\ 
				\hline 
				5& 0.00 & 0.00 & 0.00 & 0.00 & 0.00 & 0.00 & 0.00 & 0.00 & 0.00 & 0.00 & 0.00 & 0.00 & 0.00 & 0.00 & 0.00 & 5.0 \\ 
				\hline 
				6& 0.00 & 0.00 & 0.00 & 2.44 & 0.00 & 0.00 & 0.00 & 12.50 & 0.00 & 0.00 & 0.00 & 0.00 & 0.00 & 0.00 & 0.00 & 10.06 \\ 
				\hline 
				7& 0.00 & 0.00 & 7.00 & 0.00 & 0.00 & 0.00 & 0.00 & 0.00 & 15.00 & 0.00 & 0.00 & 0.00 & 0.00 & 0.00 & 0.00 & 22.00 \\ 
				\hline 
				8& 0.00 & 0.00 & 0.00 & 0.00 & 0.00 & 0.00 & 0.00 & 0.00 & 0.00 & 0.00 & 0.00 & 0.00 & 0.00 & 0.00 & 0.00 & 15.63 \\ 
				\hline 
				9& 0.00 & 0.00 & 0.00 & 0.00 & 0.00 & 0.00 & 0.00 & 0.00 & 0.00 & 0.00 & 0.00 & 0.00 & 0.00 & 0.00 & 0.00 &  42.86\\ 
				\hline 
				10& 0.00 & 0.00 & 0.00 & 0.00 & 0.00 & 0.00 & 0.00 & 0.00 & 0.00 & 0.00 & 0.00 & 0.00 & 0.00 & 0.00 & 0.00 & 4.00 \\ 
				\hline 
				11& 0.00 & 0.00 & 33.33 & 0.00 & 0.00 & 0.00 & 0.00 & 0.00 & 0.00 & 0.00 & 0.00 & 0.00 & 0.00 & 0.00 & 0.00 &0.00  \\ 
				\hline 
				12& 0.00 & 0.00 & 36.00 & 0.00 & 0.00 & 0.00 & 44.00 & 0.00 & 0.00 & 0.00 & 0.00 & 0.00 & 0.00 & 0.00 & 0.00 &0.00  \\ 
				\hline 
				13& 0.00 & 0.00 & 6.67 & 33.33 & 0.00 & 0.00 & 0.00 & 0.00 & 0.00 & 0.00 & 0.00 & 26.67 & 0.00 & 0.00 & 0.00 & 0.00 \\ 
				\hline 
				14& 0.00 & 0.00 & 0.00 & 0.00 & 0.00 & 0.00 & 0.00 & 0.00 & 0.00 & 0.00 & 0.00 & 0.00 & 0.00 & 0.00 & 0.00 & 40.00 \\
				\hline
				15& 0.00 & 0.00 & 0.00 & 0.00 & 0.00 & 0.00 & 0.00 & 0.00 & 0.00 & 0.00 & 0.00 & 0.00 & 0.00 & 0.00 & 0.00 & 16.25 \\ 
				\hline 
			\end{tabular} 
		}
		\caption{The values of the flux in the optimal networks, $\alpha = 0.95$ and $\varepsilon = 10^{-6}$.}
		\label{table:values_of_fluex}
	\end{table}
	
	\bibliographystyle{plain}
	\bibliography{Bib_ControlInput}


\end{document}